\documentclass[a4paper,article,12pt]{elsarticle}

\usepackage{a4wide}                      

\usepackage[small,bf,hang]{caption} 

\usepackage{graphicx}
\usepackage{subfig}

\usepackage{hyperref}

\usepackage{algorithm, algorithmic}

\usepackage{amsthm} 

\usepackage{color}

\usepackage{amsmath,amssymb}
\newcommand{\Aut}{\operatorname{Aut}}

\usepackage{multirow}

\journal{???}

\newtheorem{theorem}{Theorem}

\newtheorem{lemma}[theorem]{Lemma}
\newtheorem{conjecture}[theorem]{Conjecture}

\newtheorem{observation}[theorem]{Observation}

\newtheorem{definition}[theorem]{Definition}

\graphicspath{{figures/}}

\makeatletter
\def\ps@pprintTitle{%
 \let\@oddhead\@empty
 \let\@evenhead\@empty
 \def\@oddfoot{\centerline{\thepage}}%
 \let\@evenfoot\@oddfoot}
\makeatother

\begin{document}

\begin{frontmatter}

\title{Generation of cubic graphs and snarks with large girth}

\author[gent]{Gunnar Brinkmann}
\ead{Gunnar.Brinkmann@UGent.be}

\author[gent]{Jan Goedgebeur\fnref{fwo}}
\ead{Jan.Goedgebeur@UGent.be}

\address[gent]{Department of Applied Mathematics, Computer Science \& Statistics\\
  Ghent University\\
Krijgslaan 281-S9, \\9000 Ghent, Belgium\\}

\fntext[fwo]{Supported by a Postdoctoral Fellowship of the Research Foundation Flanders (FWO).}

\begin{abstract}
We describe two new algorithms for the generation of all
non-isomorphic cubic graphs with girth at least $k\ge 5$
which are very efficient for $5\le k \le 7$
and show how these algorithms can be efficiently restricted
to generate snarks with girth at least $k$.

Our implementation of these algorithms is more than $30$, respectively  $40$ times
faster than the previously fastest generator for cubic graphs with
girth at least $6$ and $7$, respectively.

Using these generators we have also generated all non-isomorphic
snarks with girth at least $6$ up to $38$ vertices and show that
there are no snarks with girth at least $7$ up to $42$ vertices.  
We present and analyse the new
list of snarks with girth 6. 

\end{abstract}


\end{frontmatter}


\section{Introduction}
\label{section:intro}

A cubic (or $3$-regular) graph is a graph where every vertex has degree
3. Cubic graphs have interesting applications in chemistry as they can
be used to represent molecules (where the vertices are e.g.\ carbon
atoms such as in fullerenes~\cite{kroto_85}).  Cubic graphs are also
especially interesting in mathematics since for many open problems in
graph theory it has been proven that cubic graphs are the smallest
possible potential counterexamples, i.e.\ that if the conjecture is
false, the smallest counterexample must be a cubic graph. For examples, see~\cite{snark-paper}.

For most problems the possible counterexamples can be further
restricted to the subclass of cubic graphs which are not
$3$-edge-colourable.  Note that by Vizing's classical theorem the
\textit{chromatic index} (i.e.\ the minimum number of colours required
for a proper edge-colouring of a given graph) of a cubic graph must be
3 or 4. Therefore cubic graphs with chromatic index 3 are often called
\textit{colourable} and those with chromatic index 4
\textit{uncolourable}. For most conjectures uncolourable cubic graphs
with low connectivity or girth (i.e.\ the length of the shortest cycle
of the graph) which are counterexamples can be reduced to
smaller counterexamples by certain standard operations. So in order to
avoid these trivial cases, the class of potential minimal
counterexamples is usually further restricted to the class of
\textit{snarks}: cyclically 4-edge-connected uncolourable cubic graphs
with girth at least 5. (Recall that a graph is called
\textit{cyclically $k$-edge-connected} if the deletion of fewer than
$k$ edges does not disconnect the graph into components which each
contain at least one cycle).

Several algorithms have already been developed for the generation of
complete lists of cubic graphs and it can be considered as a benchmark
problem in structure enumeration. The first complete lists of cubic
graphs were already determined at the end of the 19th century by de
Vries who determined all cubic graphs up to 10
vertices~\cite{de_vries_1889, de_vries_1891}.

The first computer approach dates from 1966 and was done by Ballaban
who generated all cubic graphs up to 12 vertices~\cite{balaban_66}. In
the following decades several other algorithms for generating all
cubic graphs have been proposed, each of them faster than the previous
algorithm. For a more detailed overview of the history of the
generation of cubic graphs, we refer to~\cite{brinkmann2013history}.

The fastest algorithm up until recently was developed by the first
author in 1992~\cite{brinkmann_96} and it has been implemented in a
program called \textit{minibaum}. At that time it could be used to
generate all (connected) cubic graphs up to 24 vertices and in the
meantime (when faster computers became available) it has already been
used to generate all cubic graphs up to 30 vertices.

In 1999 Meringer~\cite{meringer_99} published an efficient algorithm
for the generation of general regular graphs (so not just $3$-regular
graphs). This generator has been implemented in a program called
\textit{genreg}, but it is slower than \textit{minibaum} for
generating $3$-regular graphs.

Together with McKay~\cite{brinkmann_11} we presented a new
construction algorithm for cubic graphs in 2011. This algorithm has
been implemented in a program called \textit{snarkhunter} and is
significantly faster than \textit{minibaum}.

The generation algorithm of \textit{snarkhunter} was also adapted to
efficiently restrict the generation to cubic graphs with girth at
least 4 or 5. Also in these cases snarkhunter was faster than
\textit{minibaum}. The bounding criteria incorporated in
\textit{snarkhunter} can in principle also be applied to generate
cubic graphs with larger lower bounds on the girth, but this pruning
technique is likely to become much less efficient for larger girth.

Before~\cite{snark-paper}, the only way to generate complete lists of snarks was
to use a generator for all cubic graphs with girth at least $5$ and
apply a filter for colourability and cyclic edge-connectivity. However
as only a tiny fraction of the cubic graphs is uncolourable
(e.g.\ only 0.00008\% of the cubic graphs with girth at least $5$ and $32$ vertices are
snarks), this is certainly not an efficient approach.

The generation algorithm of \textit{snarkhunter} allowed a look-ahead
which in many cases avoids the generation of cubic graphs which are
$3$-edge-colourable. By using this specialised generation algorithm for
snarks, together with H{\"a}gglund and
Markstr{\"o}m~\cite{snark-paper} we were able to generate all snarks
up to $36$ vertices. Using these new lists of snarks, we tested several
open conjectures and were able to refute 8 of them. So this shows that
snarks are not only theoretically a good source for counterexamples to
conjectures, but also in practice.

For several conjectures the girth condition for possible minimal
counterexamples has been strengthened (e.g.\ for the cycle double cover
conjecture Huck~\cite{huck2000reducible} has even proven that a smallest 
possible counterexample must be a snark with girth at least
$12$). Furthermore even though only a tiny
fraction of the cubic graphs are snarks, there are already $60~167~732$
snarks with $36$ vertices and testing conjectures on all of these snarks
can be computationally very expensive or even infeasible. So it would
be interesting to consider more restricted subclasses of snarks that
still have the property that smallest
possible counterexamples for a lot of conjectures are contained in it. Therefore the main
focus of this paper is the generation of cubic graphs and snarks with
girth at least $6$ or $7$.

In Section~\ref{section:generation_algorithm} we describe two new
algorithms for the generation of all non-isomorphic
(connected) cubic graphs with girth at least $k\ge 5$ which are very efficient for
$5\le k\le 7$. In the
same section we also show how both algorithms can be extended with
look-aheads which often allow to avoid the generation of
$3$-edge-colourable graphs such that snarks with girth at least $k$ can
be generated more efficiently.

In Section~\ref{subsect:running_times} we report the results and running times of
our implementation of these algorithms for $5\le k\le 7$. Our
implementation is more than $30$, respectively $40$ times faster than the
previously fastest generator for cubic graphs with girth at least $6$
and $7$, respectively. For $k=5$ the algorithm which we developed
earlier jointly with McKay in~\cite{brinkmann_11} was already very
efficient, so in this case the speedup of the new algorithm is smaller (about $40$\%).

The generation algorithms can also be applied for $k>7$,  
but are less efficient for these cases. Therefore we only implemented these algorithms for $k \le 7$. For the generation of 
cubic graphs with girth $k\ge 9$
the fastest method is that of McKay, Myrvold and
Nadon~\cite{mckay1998fast}.

In~\cite{snark-paper} it was already determined that there are only 3
snarks with girth at least $6$ up to $36$ vertices. The new generator
enabled us to generate all non-isomorphic snarks with girth at least
$6$ and $38$ vertices. There are exactly 39 such snarks and in
Section~\ref{subsect:results_snarks} we present and analyse these
snarks.

We also generated a sample of snarks (i.e.\ girth at least $5$) with
$38$ vertices and a sample of snarks with girth at least $6$ and $40$
vertices. In these cases it was computationally infeasible to
generate the complete sets. In Section~\ref{subsect:results_snarks} we also describe which conjectures we tested on these new lists of snarks.

Jaeger and Swart conjectured that all snarks have girth at most
6~\cite{JS1980}. Kochol disproved this conjecture by giving a method
for constructing snarks with arbitrarily large
girth~\cite{kochol1996snarks}. However the smallest snark with girth
$7$ is currently still unknown. The smallest snark with girth $7$
obtained by Kochol's method has 298 vertices, but there is no reason
to assume that this is the smallest snark with girth $7$.

Using the new generators, we show that there are no snarks with girth at least $7$ up to $42$ vertices.

\section{Generation algorithms}
\label{section:generation_algorithm}

\subsection{Introduction}
\label{subsection:intro_algo}

An \textit{insertion operation} or \textit{expansion} is an operation
which constructs a larger graph from a given graph. We call the
inverse operation a \textit{reduction}.

Figure~\ref{fig:edgeinsert_old} shows the main construction operation
used in~\cite{brinkmann_11}. Note that this operation adds 2 new
vertices. As already mentioned in~\cite{brinkmann_11}, this algorithm
can easily be equipped with look-aheads to generate graphs with a
given lower bound $k$ on the girth efficiently. This was implemented
in the program \textit{snarkhunter}~\cite{snarkhunter-site} for
$k=4$ and $5$. For $k>5$ this method would not be efficient.

\begin{figure}[h!t]
	\centering
	\includegraphics[width=0.7\textwidth]{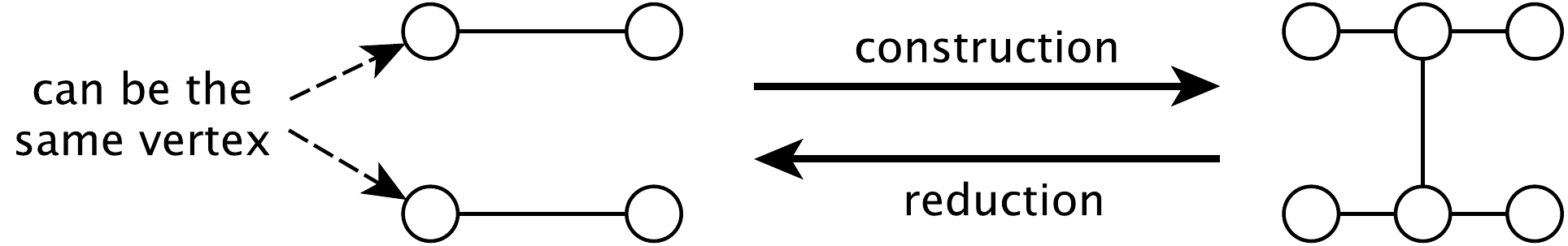}
	\caption{The basic edge insertion operation for cubic graphs used in~\cite{brinkmann_11}.}
	\label{fig:edgeinsert_old}
\end{figure}

As the number of cubic graphs grows very fast (see e.g.~\cite{robinson_83}),
in algorithms which recursively construct cubic graphs from smaller
cubic graphs, the most time consuming step is nearly always the last step.

Therefore in this paper we present two new algorithms which apply a
special operation in the last step.  The first algorithm which is
described in Section~\ref{subsect:tripod} uses an operation that adds
$4$ new vertices. We call it the \textit{tripod operation}. The
second algorithm which is described in Section~\ref{subsect:h-oper}
uses an operation that adds $6$ new vertices and we call it the
\textit{$\mathcal{H}$ operation}.

In both the tripod and $\mathcal{H}$ algorithm we use the canonical
construction path method~\cite{mckay_98} to make sure that only
pairwise non-isomorphic cubic graphs are generated. In order to use
this method we have to define a \textit{canonical reduction} which is
unique up to isomorphism for every graph we want to generate. We call
the expansion that is the inverse of a canonical reduction a
\textit{canonical expansion}. The two rules of the canonical
construction path method are:

\begin{enumerate}

\item Only accept a graph if it was constructed by a canonical expansion.

\item For every graph $G$ to which construction operations are
  applied, only perform one expansion from each equivalence class of
  expansions of $G$.

\end{enumerate}

The pseudocode for a general generation algorithm which uses the canonical
construction path method is shown in
Algorithm~\ref{algo:generate_general_graphs}.

\begin{algorithm}[ht!]
\caption{Construct(graph $G$)}
  \begin{algorithmic}
  \label{algo:generate_general_graphs}
	\IF{$G$ has the desired number of vertices}  
  		\STATE output $G$
  	\ELSE
  		\STATE find expansions
  		\STATE compute equivalence classes of expansions
  		\FOR{each equivalence class}
  			\STATE choose one expansion $X$
			\STATE perform expansion $X$
			\IF{expansion $X$ was canonical}  
				\STATE Construct(expanded graph)
			\ENDIF
                 \STATE perform reduction $X^{-1}$
		\ENDFOR
  	\ENDIF
  \end{algorithmic}
\end{algorithm}

In Section~\ref{subsect:tripod} and~\ref{subsect:h-oper} we will
describe how we apply the canonical construction path method for the
tripod and $\mathcal{H}$ algorithm, respectively and prove that
exactly one representative of every isomorphism class of connected
cubic graphs with girth at least $k$ is generated.

In those sections we will also show how both algorithms can be extended to generate snarks with girth at least $k$ significantly more efficiently than just filtering.

\subsection{Tripod operation}
\label{subsect:tripod}

\subsubsection{Tripod operation for generating cubic graphs}

Figure~\ref{fig:tripod_operation} shows the tripod operation. This
operation subdivides the edges of a set of $3$ different edges 
(which we call an \textit{edge-triple}) and
inserts a $K_{1,3}$ between them. The edges of the edge-triple must be distinct, but they can have common vertices.
This construction operation yields a
cubic graph with $n+4$ vertices when applied to an edge-triple
of a cubic graph with $n$ vertices.

We call the $K_{1,3}$ inserted by the tripod insertion operation a
\textit{tripod} and we call the vertex of degree 3 in the $K_{1,3}$
the \textit{central vertex} of the tripod. The central vertex is
labelled $c$ in Figure~\ref{fig:tripod_operation}. Note that the
central vertex of a tripod uniquely identifies the tripod.

\begin{figure}[h!t]
	\centering
	\includegraphics[width=0.9\textwidth]{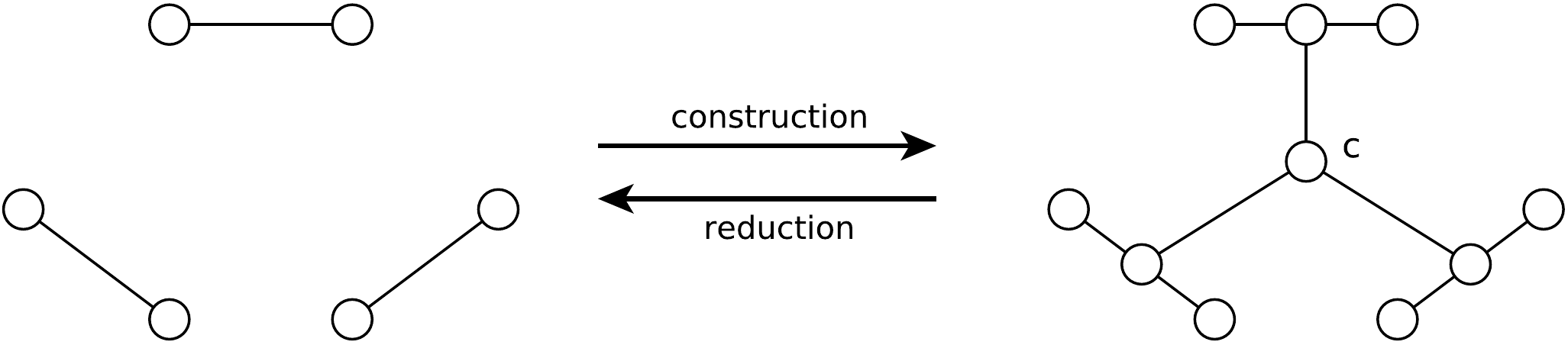}
	\caption{The tripod construction and reduction operation.}
	\label{fig:tripod_operation}
\end{figure}

We call a tripod \textit{reducible} if the central vertex of the tripod is not a cut vertex and each cycle with $i$ endvertices of the tripod has length at least $3+i$.

We
define the \textit{edge-distance} $d_e(e_1,e_2)$ between $2$ edges
as the number of edges in a shortest path
containing $e_1$ and $e_2$. 
The
\textit{minimum edge-triple distance} $d_t(e_1,e_2,e_3)$ of an edge-triple $\{e_1,e_2,e_3\}$ is defined as $\min\{d_e(e_1,e_2),
d_e(e_1,e_3), d_e(e_2,e_3)\}$.
When applying a tripod operation to an edge-triple $\{e_1,e_2,e_3\}$ that inserts a new vertex $v$, $d_e(e_i,e_j)$ is equal to the usual distance $d(v_i,v_j)$ (that is:
the number of edges in a shortest path from $v_i$ to $v_j$) in $G-v$
if $v_i$, respectively $v_j$ denote the vertices subdividing $e_i$ and $e_j$, respectively.

We say that two
subgraphs of a graph are \textit{disjoint} if they have no common
vertices.

\begin{theorem}
\label{theorem:reducible_tripod}
Every connected cubic graph $G$ with girth $k\ge 5$ has a reducible tripod.

The graph obtained by reducing a tripod has girth at least $k-1$ and
at most $3$ disjoint cycles of length $k-1$.
\end{theorem}

\begin{proof}

Every connected graph has a spanning tree and every tree with more
than one vertex has at least two leaves. So every connected graph with
at least $2$ vertices has at least $2$ vertices which are no cut vertices.

Choose an arbitrary vertex $v$ that is not a cut vertex as the central vertex of a tripod. Since $G$ has girth $k \ge 5$, each cycle with $i$ endvertices of the tripod has length at least $3+i$.
As $G$ has girth $k$, any two neighbours of $v$ have distance
at least $k-2$ in $G-v$, so any cycle in $G-v$ containing $m$ neighbours has
length at least $\max\{k,(k-2)m\}$ leading to a cycle of length at least
$\max\{k-m,(k-3)m\}\ge k-1$ in the reduced graph.

The last point follows immediately as in each cycle of length $k-1$
an edge must have been subdivided by a neighbour of $v$.

\end{proof}

So all connected cubic graphs with girth at least $k$ (for $k \ge 5$)
with $n$ vertices can be obtained by applying the tripod operation to
all possible edge-triples of connected cubic graphs with $n-4$
vertices in such a way that the tripod operation destroys all cycles
of size smaller than $k$ without creating any new cycles of size
smaller than $k$.

The following lemma follows directly from the definition of
the tripod operation.

\begin{lemma}
\label{lemma:tripod_distance}
When applying the tripod expansion operation to the edge-triple
$\{e_1,e_2,e_3\}$ of a graph $G$, the smallest cycle 
containing the central vertex of the
tripod in the expanded graph has size $d_t(e_1,e_2,e_3) +
2$.
\end{lemma}

So more specifically, all cubic graphs with girth at least $k$ (for $k \ge
5)$ can be obtained by applying the tripod expansion operation to all
possible edge-triples $\{e_1,e_2,e_3\}$ of cubic graphs with girth at
least $k-1$ with at most 3 disjoint cycles of size $k-1$ for which
$d_t(e_1,e_2,e_3) \ge k-2$ and which destroy all cycles of size
$k-1$. We call such edge-triples \textit{eligible} edge-triples.

Now we still have to make sure that our algorithm does not output any
isomorphic copies. We apply the canonical construction path
method~\cite{mckay_98} to do this. We will now define what 
a canonical tripod expansion is and when expansions are considered equivalent.
Then the two rules of the canonical
construction path method applied to the tripod operation algorithm are as follows.

\begin{enumerate}

\item Only accept a cubic graph with girth at least $k$ if it was constructed by a canonical tripod expansion.

\item For every graph $G$ to which construction operations are
  applied, only perform one tripod expansion from each equivalence
  class of expansions of $G$.

\end{enumerate}

Two edge-triples $\{e_1,e_2,e_3\}$ and $\{e_1',e_2',e_3'\}$ of $G$ are
called \textit{equivalent} if there is an automorphism of $G$ mapping
$\{e_1,e_2,e_3\}$ to $\{e_1',e_2',e_3'\}$, i.e.\ if they are in the
same orbit of edge-triples under the action of the automorphism group
of $G$. We also denote the automorphism group of $G$ as $\Aut(G)$.

Recall from Section~\ref{subsection:intro_algo} that a
\textit{canonical expansion} is the inverse of a canonical reduction
and that we have to define a \textit{canonical reduction} which is
unique up to isomorphism.

To define an efficient criterion for a tripod reduction of a graph $G$
to be canonical, we
assign a 5-tuple $(x_0,\dots ,x_4)$ to the central vertex $c$ of every
reducible tripod and define an operation reducing
the tripod with the smallest
5-tuple as the canonical reduction.  The values of $x_0,\ldots,x_3$
are combinatorial invariants of increasing discriminating power and
cost and are defined as follows:

\begin{enumerate}

\item $x_0$ is the number of neighbours of $c$ which are part of a pentagon which does not contain $c$.

\item $x_1$ is the number of neighbours of $c$ which are part of a hexagon which does not contain $c$.

\item $x_2$ is the number of vertices at distance at most $3$ from $c$.

\item $x_3$ is the number of vertices at distance at most $4$ from $c$.

\end{enumerate}

We chose these values of $x_i$ based on several performance tests comparing different choices for $x_i$.

Clearly, $x_0=0$ for all reducible tripods when generating graphs with girth at least $6$,
so it does not discriminate between tripods and does not have to be computed.
Analogously, $x_0,x_1$ and $x_2$ do not have to be computed when generating graphs with girth at least $7$. 

We call the tripod which was inserted by the last tripod insertion
operation the \textit{inserted tripod}. The value of $x_i$ only has to
be computed for the reducible tripods which have the smallest value of
$(x_0,...,x_{i-1})$. Since we only need to know if the inserted
tripod is canonical, we can stop as soon as we have found a reducible
tripod with a smaller tuple $(x_0,...,x_{i-1})$ (which implies that
the inserted tripod is not canonical). We can also stop in case the
inserted tripod is the only reducible tripod with the smallest value
of $(x_0,...,x_{i-1})$ (which implies that the inserted tripod is
canonical).

Only in case there are multiple reducible tripods with the smallest
value of $(x_0,...,x_3)$ (and the inserted tripod is one of them), we
use the program \textit{nauty}~\cite{nauty-website, mckay_14} to
compute a canonical labelling of the graph and define $x_4$ to be the
smallest label in the canonical labelling of $G$ of a vertex which is
in the same orbit of $\Aut(G)$ as $c$.

The discriminating power of $(x_0,...,x_3)$ is usually enough to
decide whether or not the reduction is canonical, so the more
expensive computation of $x_4$ can often be avoided. For example for
generating connected cubic graphs with girth at least $6$ and $34$
vertices, the computation of $x_4$ is only required in about $4$\% of
the cases.

The definitions of $x_0$ and $x_1$ also allow look-aheads on level
$n-4$ for deciding whether or not an expansion can be canonical before
actually performing the tripod expansion operation. This allows to
avoid the construction of a lot of graphs that would not be canonical.
For example for generating connected cubic graphs with girth at least
$6$ and $n$ vertices, these look-aheads allow to avoid the
construction of about $89$\% of the graphs for $n=34$. This
percentage seems to be increasing for larger $n$.

Note that two vertices $v_1,v_2$ have the same value for
$(x_0,...,x_4)$ if and only if they are in the same orbit of the
automorphism group of the graph.

Now we can prove the following theorem.

\begin{theorem}\label{lem:exactlyone_tripod}
Assume that exactly one representative of each isomorphism class of
connected cubic graphs with $n-4$ vertices with girth at least $k-1$
and at most 3 disjoint cycles of size $k-1$ is given.  If we apply the
following steps:

\begin{enumerate}\itemsep=0pt
\item Apply the tripod insertion operation to one edge-triple in each orbit of eligible edge-triples.

\item Accept each graph with $n$ vertices if and only if the inserted
  tripod has minimal value of $(x_0,\dots ,x_4)$ among all possible
  tripod reductions of the graph.
\end{enumerate}

Then exactly one representative of each isomorphism
class of connected cubic graphs with $n$ vertices and girth at least $k$ is accepted.
\end{theorem}

\begin{proof}
First we will show that at least one representative of each
isomorphism class of connected cubic graphs with $n$ vertices and
girth at least $k$ is accepted.

Let $G$ be a connected cubic graph with $n$ vertices and girth at
least $k$. It follows from Theorem~\ref{theorem:reducible_tripod} that
$G$ has a reducible tripod, so also a canonical tripod reduction. Let
$c$ be the central vertex of a canonical tripod of $G$ and we call the
graph obtained by applying the tripod reduction to $c$ $p(G)$. The
tripod of $G$ which was reduced corresponds to an edge-triple
$\{e_1,e_2,e_3\}$ in $p(G)$.

By Theorem~\ref{theorem:reducible_tripod} $p(G)$ has girth at least $k-1$
and contains at most $3$ disjoint cyles of size $k-1$, so $p(G)$ is
isomorphic to one of the input graphs $H$ to which the algorithm was
applied. Let $\gamma$ be an isomorphism from $p(G)$ to $H$.  The graph
$H$ has an eligible edge-triple which is in the same orbit of
edge-triples as $\gamma(\{e_1,e_2,e_3\})$ and the algorithm applies
the tripod insertion operation to such an edge-triple. Let $c'$ be the
central vertex of the tripod obtained by applying the tripod expansion
operation to this edge-triple in $H$. This produces a graph $G'$ which
is isomorphic to $G$, with an isomorphism $\gamma^*$ from $G$
to $G'$ mapping $c$ to $c'$. This implies that
$c'$ has a minimal value of $(x_0,...,x_4)$ and thus $G'$ is accepted by
the algorithm.

\bigskip

Now we will show that at most one representative of each isomorphism
class of connected cubic graphs with $n$ vertices and girth at least
$k$ is accepted.

Suppose the algorithm accepts two isomorphic graphs $G$ and $G'$. Call
the graphs obtained by applying a canonical tripod reduction to $G$
and $G'$ $p(G)$ and $p(G')$, respectively. Let $t$ be the canonical
tripod which was reduced in $G$ and $t'$ be the canonical tripod which
was reduced in $G'$ and let $\gamma$ be an isomorphism from $G$ to
$G'$. Since $t$ and $t'$ are both canonical tripods, $\gamma(t)$ is in
the same orbit as $t'$ under the action of $\Aut(G')$, so we can assume
that $\gamma$ maps $t$ onto $t'$. But this automorphism induces
an isomorphism $\gamma_p$ from $p(G)$ to $p(G')$. So by our assumption this 
means that $p(G)$ and $p(G')$ are identical and that $\gamma_p$ is an automorphism
mapping the edge-triple that was extended to form $G$ to the edge-triple that was
extended to form $G'$. This implies that equivalent edge-triples were extended
-- contradicting our assumption.

\end{proof}

To generate the connected cubic graphs with $n-4$ vertices with girth
at least $k-1$, we used the efficient generator for cubic graphs
described in~\cite{brinkmann_11} in case $k \le 5$ or recursively
apply the algorithm from Theorem~\ref{lem:exactlyone_tripod} in case
$k > 5$.

It is also possible to incorporate look-aheads to avoid the generation of
connected cubic graphs with $n-4$ vertices with girth at least $k-1$
with more than 3 disjoint cycles of size $k-1$. But in practice, this
is not very helpful as the generation of cubic graphs with $n-4$
vertices and girth at least $k-1$ is usually not a bottleneck.

\subsubsection{Tripod operation for generating snarks}

In this section we describe how a look-ahead for the tripod operation
can be used which often allows to avoid the generation of
$3$-edge-colourable graphs.

Note that the application of the edge tripod insertion operation to an
edge-triple can also be seen as first applying the edge insertion
operation from Figure~\ref{fig:edgeinsert_old} to two of these edges
and then applying the edge insertion operation to the new edge and the
third edge.

Given a cubic graph $G$ and let $c: E(G)\to \{1,2,3\}$ be a proper
$3$-edge-colouring of $G$. We call a cycle of $G$ where all edges have
one of two colours $c_1, c_2$ in the $3$-edge-colouring $c$ of $G$ a
\textit{colour cycle} or more specifically a
\textit{$(c_1,c_2)$-cycle}.

In~\cite{snark-paper}, jointly with H{\"a}gglund and Markstr{\"o}m we
used the following well-known fact to avoid the generation of a lot of
$3$-edge-colourable graphs using the edge insertion operation:

\begin{theorem}\label{theorem_add_edge_cycle}
	Given a cubic graph $G$ and a $3$-edge-colouring of $G$. If
        two edges $e,e'$ belong to the same colour cycle
(or equivalently: there is a path containing $e,e'$ using only edges
of two different colours), the graph
        $G'$ obtained by applying the edge insertion operation from
        Figure~\ref{fig:edgeinsert_old} to $e,e'$ will be
        $3$-edge-colourable.

\end{theorem}

\begin{definition}
Let $\{e_1,e_2,e_3\}$ be an edge-triple in a $3$-edge-coloured cubic graph
and $\{c_1,c_2,c_3\}= \{1,2,3\}$.

A $(c_1,c_2)$-path from
an edge $e_1$ coloured with $c_1$ to an edge $e_2$ is a path starting with 
$e_1$ and ending with $e_2$, that contains only edges with colours
$c_1$ and $c_2$. Note that the colour of $e_2$ can be both
$c_1$ and $c_2$.

\begin{itemize}
\item A \underline{prune-path} for $\{e_1,e_2,e_3\}$ is a path from $e_1$ to $e_3$ containing
$e_2$, so that the path contains a
$(c_2,c_1)$-path from $e_2$ to $e_1$ and a $(c_2,c_3)$-path from $e_2$ to $e_3$.

\item Let $T$ be a tree with $3$ leaves contained in
the edges $\{e_1,e_2,e_3\}$ so that $e_1$ also contains the (unique) vertex $v$ of degree $3$ 
and is
coloured with $c_1$. If the path starting in the vertex $v$ and ending with 
$e_2$ is a $(c_2,c_1)$-path and the path starting in the vertex $v$ and ending with 
$e_3$ is a $(c_3,c_2)$-path, the tree $T$ is called a \underline{prune-tree} for $\{e_1,e_2,e_3\}$.

\end{itemize}

\end{definition}

In the following theorem we present a look-ahead for $3$-edge-colourability for the tripod insertion operation.

\begin{theorem}\label{theorem_tripod_snarks}
Given a cubic graph $G$, a $3$-edge-colouring $c$ of $G$
and edges $\{e_1,e_2,e_3\}$ in $G$. 

If there is a prune-path or a prune-tree for $\{e_1,e_2,e_3\}$, then the graph
obtained by applying a tripod operation to $\{e_1,e_2,e_3\}$ is $3$-edge-colourable.

\end{theorem}

\begin{proof}

Applying the tripod operation to $\{e_1,e_2,e_3\}$ is equivalent to first applying
the edge insertion operation from Figure~\ref{fig:edgeinsert_old} to $e_1,e_2$ and then to the new edge and $e_3$.

If there is a prune-path for $\{e_1,e_2,e_3\}$ with a $(c_2,c_1)$-path from $e_2$ to $e_1$,
then in the graph obtained by applying the edge insertion operation to $e_1,e_2$,
there is a $(c_3,c_2)$-path from the new edge to $e_3$ for the colouring shown in
Figure~\ref{fig:recolourtripod-p}.

If there is a prune-tree for $\{e_1,e_2,e_3\}$ with $e_1$ adjacent to the vertex of degree
$3$ and a $(c_1,c_2)$-path from $e_1$ to $e_2$, then there is a $(c_3,c_2)$-path from the new
edge to $e_3$ in the graph obtained by applying the edge insertion operation to $e_1,e_2$
for the colouring shown in
Figure~\ref{fig:recolourtripod-t}.

In each case the graph after applying both edge insertion operations is $3$-edge-colourable by Theorem~\ref{theorem_add_edge_cycle}.

\end{proof}

\begin{figure}[h!t]
    \centering
    \subfloat[]{
    \label{fig:recolourtripod-p-a}
    \includegraphics[width=0.75\textwidth]{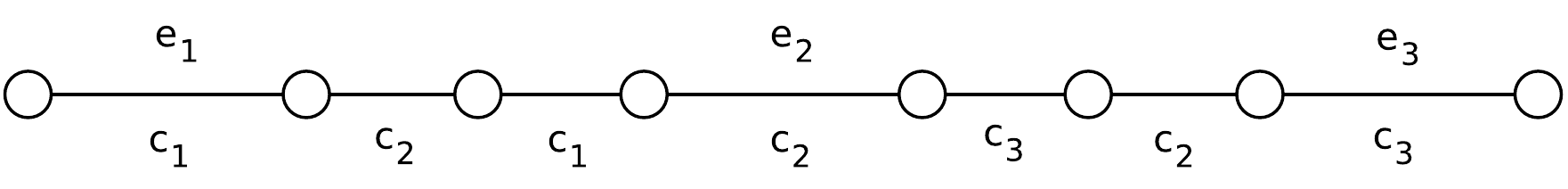}
    }

    \subfloat[]{
    \label{fig:recolourtripod-p-b}
    \includegraphics[width=0.75\textwidth]{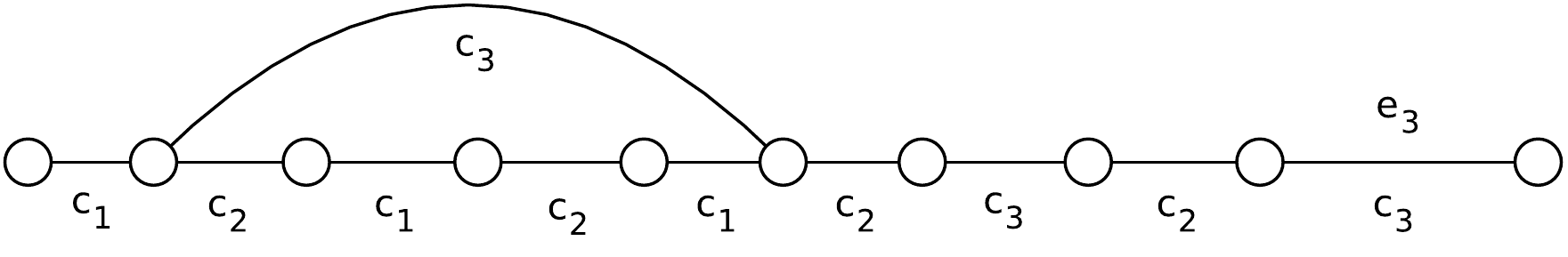}
    }    
    \caption{Procedure to recolour a prune-path when applying the tripod insertion operation to an edge-triple $\{e_1,e_2,e_3\}$ used in the proof of Theorem~\ref{theorem_tripod_snarks}.}
    \label{fig:recolourtripod-p}
\end{figure}

\begin{figure}[h!t]
    \centering
    \subfloat[]{
    \label{fig:recolourtripod-t-a}
    \includegraphics[width=0.75\textwidth]{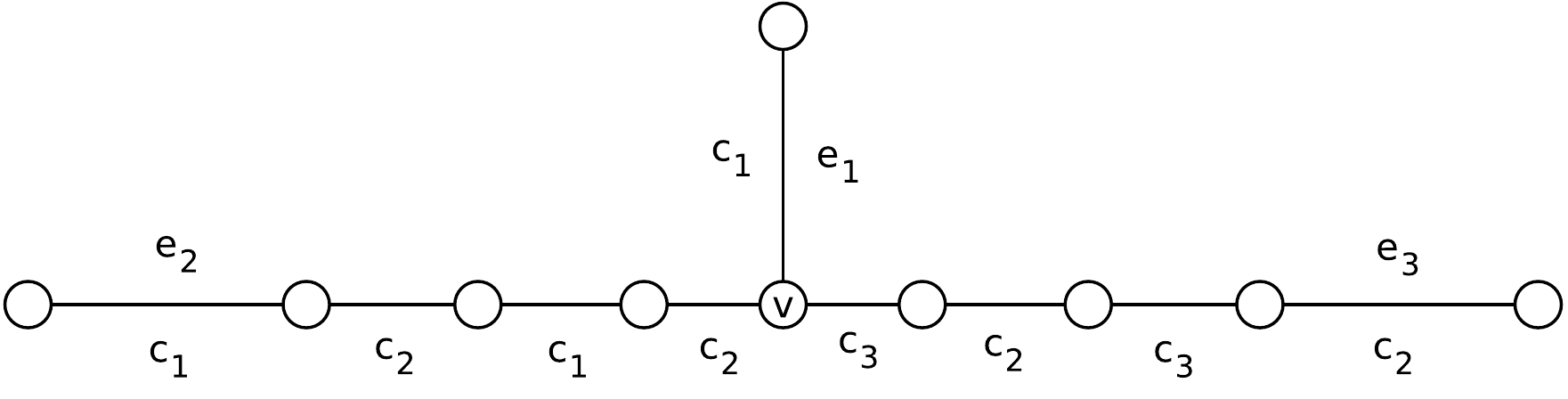}
    }

    \subfloat[]{
    \label{fig:recolourtripod-t-b}
    \includegraphics[width=0.75\textwidth]{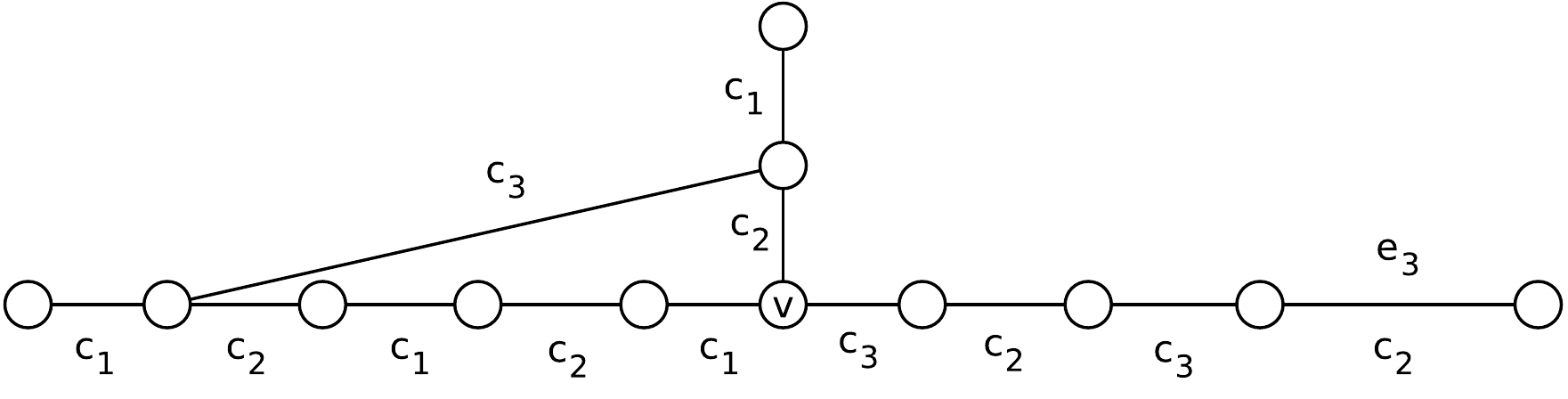}
    }    
    \caption{Procedure to recolour a prune-tree when applying the tripod insertion operation to an edge-triple $\{e_1,e_2,e_3\}$ used in the proof of Theorem~\ref{theorem_tripod_snarks}.}
    \label{fig:recolourtripod-t}
\end{figure}

We apply this look-ahead as follows for generating snarks with $n$
vertices and girth at least $k$. When a cubic graph $G$ with $n-4$
vertices and girth at least $k-1$ with at most $3$ disjoint cycles of
size $k-1$ is generated, we compute three different $3$-edge-colourings
of $G$. The look-ahead is only applicable if $G$ is
$3$-edge-colourable, but this is not an issue as nearly all cubic
graphs are $3$-edge-colourable~\cite{robinson_92}.
We then mark all edge-triples which fulfil the conditions of
Theorem~\ref{theorem_add_edge_cycle} for one of the
colourings. This allows to discard about $94$\% of the edge-triples which would otherwise be expanded for $k=6$ and $n=36$.

The cost for computing a fourth colouring turned out to be higher than
the gain achieved by the additionally discarded edge-triples. Actually
after each colouring we first discard edge-triples and only compute the next
colouring if there are sufficiently many eligible
edge-triples left.

Theorem~\ref{theorem_tripod_snarks} is only applicable for the last
step (i.e.\ the tripod operation), but as the number of graphs grows
very fast, this already gives a considerable speedup as will be seen in
Section~\ref{section:results}.

\subsection{$\mathcal{H}$ operation}
\label{subsect:h-oper}

\subsubsection{$\mathcal{H}$ operation for generating cubic graphs}

Figure~\ref{fig:h_operation} shows the $\mathcal{H}$ operation. 
We call a set $\{\{e_1,e_2\}, \{e_3,e_4\}\}$ of two sets, each containing two
edges and with all four edges pairwise different an \textit{edge-quadruple}.
This
operation subdivides the edges of an edge-quadruple $\{\{e_1,e_2\},
\{e_3,e_4\}\}$ and inserts an $\mathcal{H}$ graph between them. An $\mathcal{H}$ graph is a tree of 6 vertices with 2 adjacent vertices of degree~3. This
construction operation yields a cubic graph with $n+6$ vertices when it
is applied to an edge-quadruple of a cubic graph with $n$
vertices.

We call the edge where both vertices have degree 3 in the
$\mathcal{H}$ graph the \textit{central edge} of the
$\mathcal{H}$. The central edge is labelled $e_c$ in
Figure~\ref{fig:h_operation}. Note that the central edge of an
$\mathcal{H}$ uniquely identifies the $\mathcal{H}$.

\begin{figure}[h!t]
	\centering
	\includegraphics[width=0.9\textwidth]{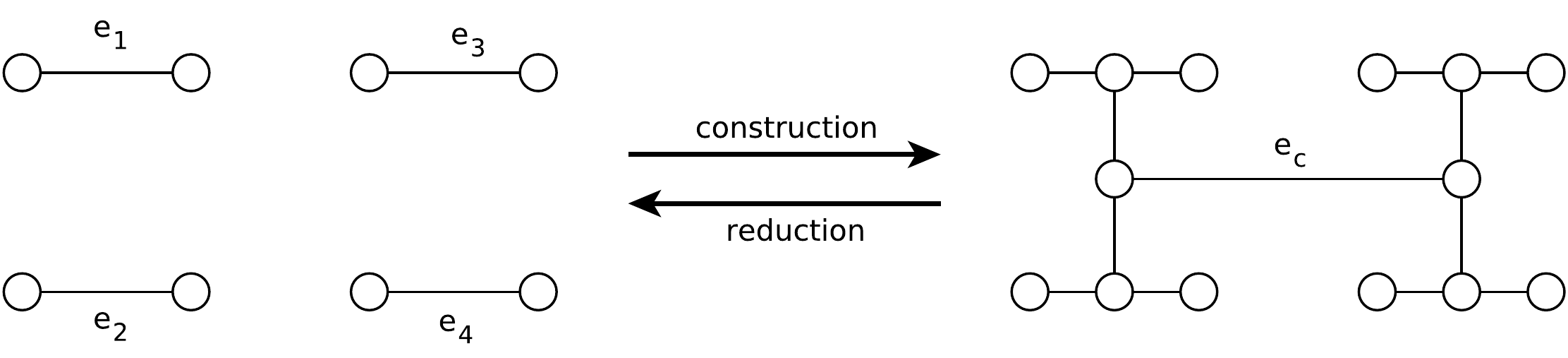}
	\caption{The $\mathcal{H}$ construction and reduction operation.}
	\label{fig:h_operation}
\end{figure}

Let $G[S]$ denote the subgraph of a graph $G$ induced by the set of
vertices $S \subseteq V(G)$. Given a connected graph $G$. We call an
edge $e = \{x,y\}$ \textit{good} if $G[V(G) - \{x,y\}]$ is connected.
We call an $\mathcal{H}$ \textit{reducible} if the central edge of the $\mathcal{H}$ is good and each cycle with $i$ endvertices of $\mathcal{H}$ has length at least $3+i$.


The following definitions and lemmas are needed for the proof that
every connected cubic graph with girth at least $5$ has a reducible
$\mathcal{H}$.

\begin{lemma}
\label{lemma:good_edge}
Every connected cubic graph has a \textit{good} edge.
\end{lemma}

\begin{proof}
De Vries~\cite{de_vries_1889, de_vries_1891} has shown that every
connected cubic graph can be obtained by recursively applying the
construction operations from Figure~\ref{fig:operations_devries} to a
$K_4$. 

$K_4$ has $6$ good edges, so it is sufficient to show that when applying the
operations to a graph with good edges, the result has good edges too.
Operation (a) adds a subgraph that contains $5$ good edges, so the result
of applying this operation always has good edges. When operation (b) or (c) are applied,
good edges not involved in the operation stay good, so assume that a good
edge is involved in the operation.

For (b) it is easy to see that in this case all newly inserted edges are good,
while for (a) both parts of the good edge that was subdivided are new good
edges.

\end{proof}

\begin{figure}[h!t]
	\centering
	\includegraphics[width=0.75\textwidth]{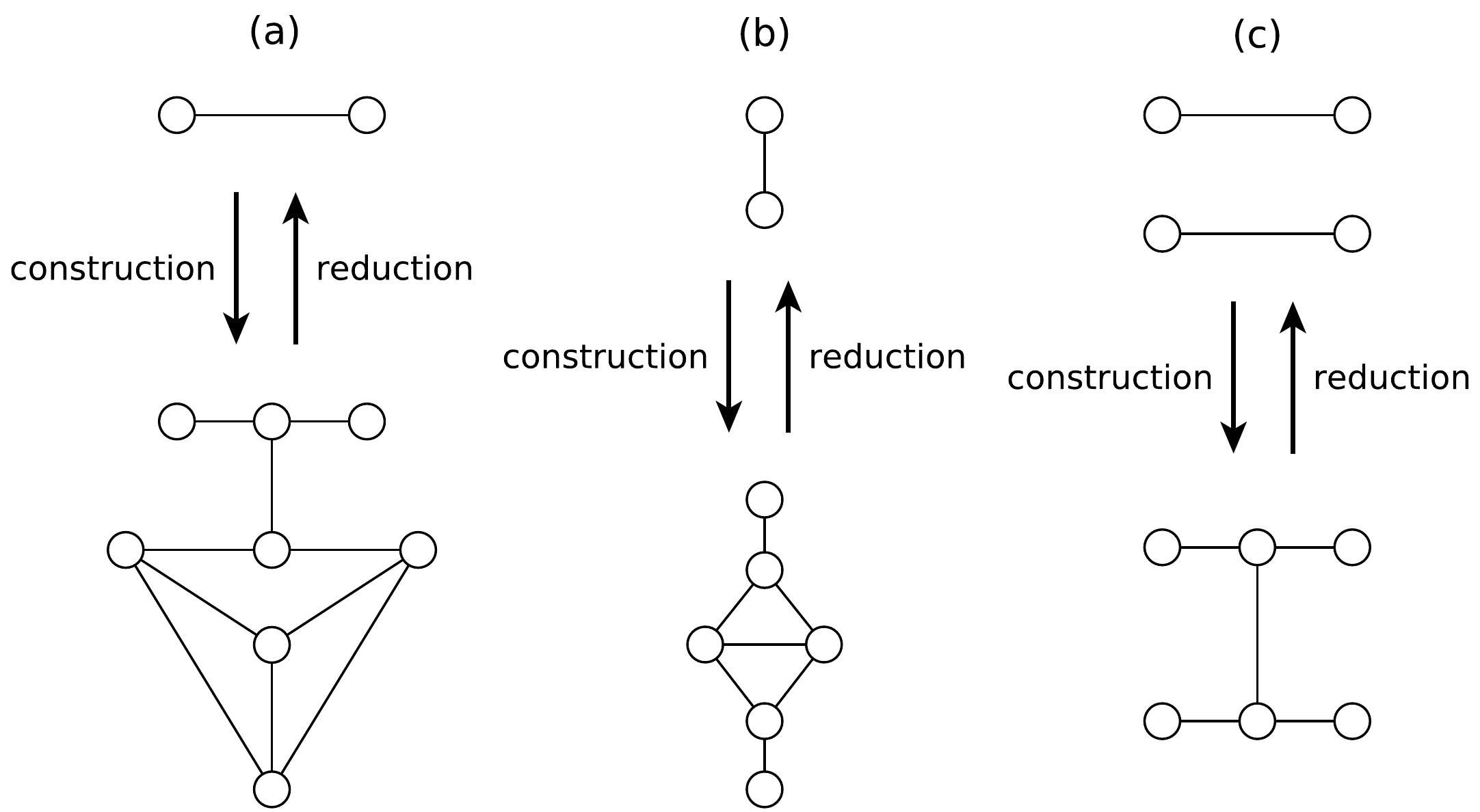}
	\caption{The construction operations for cubic graphs by de
          Vries~\cite{de_vries_1889, de_vries_1891}. Note: the
          edge-pair to which operation $(c)$  is applied is allowed to have a
          common vertex.}
	\label{fig:operations_devries}
\end{figure}

Recall from Section~\ref{subsect:tripod} that $d_e(e_1,e_2)$ denotes
the edge-distance between two edges $e_1$ and $e_2$. We define the
minimum distance $d_q(e_1,e_2,e_3,e_4)$ of an edge-quadruple 
$\{\{e_1,e_2\}, \{e_3,e_4\}\}$ as:
\begin{equation*}
min[d(e_1,e_2),d(e_3,e_4),d(e_1,e_3)+1,d(e_1,e_4)+1,d(e_2,e_3)+1,d(e_2,e_4)+1]
\end{equation*}

From the definition of the  $\mathcal{H}$ expansion we immediately get the following lemma:

\begin{lemma}
\label{lemma:h_distance}
When applying the $\mathcal{H}$ expansion operation to the
edge-quadruple $\{\{e_1,e_2\}, \{e_3,e_4\}\}$ without duplicate edges
of a graph $G$, the smallest cycle in the expanded graph which
contains an edge of the newly inserted  $\mathcal{H}$ 
has size $d_q(e_1,e_2,e_3,e_4) + 2$.
\end{lemma}

\begin{theorem}
\label{theorem:reducible_h}
Every connected cubic graph $G$ with girth at least $k\ge 5$ has a reducible $\mathcal{H}$.
The girth of the reduced graph is at least $k-2$ if $k\in \{5,6\}$ and at least $k-1$ otherwise.
\end{theorem}

\begin{proof}
Let $e=\{v,w\}$ be a good edge in $G$ and $v_1,v_2, v_3,v_4$ the vertices of the $\mathcal{H}$
not contained in $e$, so that in $\mathcal{H}$ we have that 
$d_{\mathcal{H}}(v_1,v_2)=d_{\mathcal{H}}(v_3,v_4)=2$. (Here $d_{G}(a,b)$ stands for the distance between two vertices $a$ and $b$ in the graph $G$). It follows that in the graph
$\bar G=G[V(G)-\{v,w\}]$, $d_{\bar G}(v_1,v_2)\ge k-2$, $d_{\bar G}(v_3,v_4)\ge k-2$, and
for $x\in \{v_1,v_2\}, y\in \{v_3,v_4\}$ we have $d_{\bar G}(x,y)\ge k-3$. As all these
distances are at least $2$, this also implies
that after the reduction the edges that formerly contained an endpoint
of the  $\mathcal{H}$ are pairwise different.

If a cycle in $G$ that does not contain edges of the $\mathcal{H}$ centered around
$e$,
has length $g$ and contains $i$ vertices from $\{v_1,v_2, v_3,v_4\}$,
after the reduction it will have length $g-i$, so we have to prove that 
for each such cycle $g-i\ge k-2 \ge 3$ 
to show that the result has neither loops nor double edges. Furthermore we have to show that
for $k>6$ we have that $g-i\ge k-1$.

As $g\ge 5$ the result is trivial for $i=1$ for all $k$ and for $i\le 2$ if $k\in \{5,6\}$. 
So assume just that $k>6$ and $i=2$. The distance of the two endpoints of the $\mathcal{H}$
on the cycle must be at least $k-3$, so $g\ge 2(k-3) = k + (k-6) \ge k+1$. So
$g-i=g-2\ge k-1$.

So now assume that $k \ge 5$ and $i>2$.

If $i=3$, two of the vertices on the cycle come from one of the sets $\{v_1,v_2\},\{v_3,v_4\}$,
so w.l.o.g.\ the vertices are $v_1,v_2,v_3$. This implies that 
$g \ge d_{\bar G}(v_1,v_2) + d_{\bar G}(v_1,v_3) + d_{\bar G}(v_2,v_3) \ge (k-2) + 2(k-3) = k +(2k - 8)\ge k+2$. 
So $g-i = g-3 \ge k-1$.

If $i=4$ we have all $4$ vertices on the cycle and as $d_{\bar G}(x,y)\ge 2$ for
all $x,y\in \{v_1,v_2, v_3,v_4\}$, $x\not= y$, we have $g\ge 4(k-2)= k +(3k-8) \ge k+7$ and 
$g-i = g-4 \ge k+3$.

\end{proof}

So all connected cubic graphs with girth at least $k$ (for $k \ge 5$)
with $n$ vertices can be obtained by applying the $\mathcal{H}$
operation to all possible edge-quadruples of connected cubic
graphs with $n-6$ vertices in such a way that the $\mathcal{H}$
operation destroys all cycles of size smaller than $k$ without
creating any new cycles of size smaller than $k$. We call such edge-quadruples \textit{eligible} edge-quadruples.

So in order for $\{\{e_1,e_2\}, \{e_3,e_4\}\}$ to be an eligible edge-quadruple
for generating cubic graphs with girth at least $k$,
$d_q(e_1,e_2,e_3,e_4)$ must be at least $k-2$.

Assume our aim is to generate cubic graphs with girth at least $k$. We
define the \textit{deficit} $d(k,C)$ of a cycle $C$ in a graph $G$ as
$max(k-|C|,0)$. We define the deficit $d(k,G)$ of a graph $G$ as:

\begin{equation*}
max(\{\sum_{c \in C} d(k,c) \ | \  C\ \text{is a set of disjoint cycles of } G \})
\end{equation*}

\begin{lemma}
\label{lemma:reduced_h}
When applying the $\mathcal{H}$ reduction operation to a cubic graph
with girth at least $k\ge 5$, the reduced graph $G$ has deficit $d(k,G)\le 4$.

\end{lemma}

\begin{proof}

The fact that in a cycle of length $k-i$ at least $i$ edges must be intersected by endpoints
of the $\mathcal{H}$ operation immediately implies that $d(k,G)\le 4$.

\end{proof}

As in Section~\ref{subsect:tripod}, we also use the canonical
construction path method~\cite{mckay_98} to make sure that our
algorithm does not output any isomorphic copies. The two rules of the
canonical construction path method applied to the $\mathcal{H}$
operation algorithm are:

\begin{enumerate}

\item Only accept a cubic graph with girth at least $k$ if it was constructed by a canonical $\mathcal{H}$ expansion.

\item For every graph $G$ to which construction operations are
  applied, only perform one $\mathcal{H}$ expansion from each
  equivalence class of expansions of $G$.

\end{enumerate}

Two edge-quadruples $\{\{e_1,e_2\}, \{e_3,e_4\}\}$ and
$\{\{e_1',e_2'\}, \{e_3',e_4'\}\}$ of $G$  are called \textit{equivalent} if
there is an automorphism of $G$ mapping the set $\{\{e_1,e_2\}, \{e_3,e_4\}\}$
to the set $\{\{e_1',e_2'\}, \{e_3',e_4'\}\}$.

To define an efficient criterion for an $\mathcal{H}$ reduction of a graph
$G$ to be canonical, we assign a $7$-tuple $(x_0,\dots ,x_6)$ to the central edge $e_c$
of every reducible $\mathcal{H}$ and define a reducible $\mathcal{H}$
with the smallest $7$-tuple as the canonical reduction.

We did not implement the $\mathcal{H}$ operation for $k=5$ as then the
edge-quadruples might contain incident edges, which would result in a large number
of possibly eligible edge-quadruples, so that we did not expect an advantage over the
tripod operation.

The values of $x_0,\ldots,x_4$ are combinatorial invariants of increasing
discriminating power and cost and are defined as follows:

\begin{enumerate}

\item $x_0$ is the number of hexagons containing $e_c$.

\item $x_1$ is the number of neighbours of $e_c$ which are part of a hexagon which does not contain $e_c$.

\item $x_2$ is the number of heptagons containing $e_c$.

\item $x_3$ is the number of vertices at distance at most $3$ from a vertex in $e_c$.

\item $x_4$ is the number of vertices at distance at most $4$ from a vertex in $e_c$.

\end{enumerate}

Clearly, $x_0$ and $x_1$ are constant when generating graphs with girth at least $7$, so in that case they do
not have to be computed.

Only if there are multiple reducible $\mathcal{H}$'s with the smallest
value of $(x_0,...,x_4)$ (and the inserted $\mathcal{H}$ is one of
them), we use \textit{nauty}~\cite{nauty-website, mckay_14} to compute
a canonical labeling. The canonical vertex labeling induces a
lexicographic ordering of the edges and we define $x_5, x_6$ with $x_5<x_6$
to be the labels of the lexicographically smallest edge in the same orbit of
$\Aut(G)$ as $e_c$.

The discriminating power of $(x_0,...,x_4)$ is usually enough to
decide whether or not the reduction is canonical so the more expensive
computation of $x_5,x_6$ can often be avoided. For example for
generating connected cubic graphs with girth at least $6$ and $34$
vertices, the computation of $x_5,x_6$ is only required in about $3$\% of
the cases.

Also note that the definitions of $x_0$ or $x_2$ allow look-aheads on level $n-6$ for deciding whether or
not an expansion can be canonical before actually performing the
$\mathcal{H}$ expansion operation.  For example for generating
connected cubic graphs with girth at least $6$ and $34$ vertices, these
look-aheads allow to avoid the construction of about $85$\% of the
graphs.

We omit the proof of the following theorem as it is completely
analogous to the proof of Theorem~\ref{lem:exactlyone_tripod}.

\begin{theorem}
Assume that exactly one representative $G$ of each isomorphism class of
connected cubic graphs with $n-6$ vertices with deficit $d(k,G)$ at most 4 and
girth at least $k-2$ (in case of $k\in \{5,6\}$) or girth at least $k-1$
(in case of $k \ge 7)$ is given.  If we apply the following steps:

\begin{enumerate}\itemsep=0pt
\item Apply the $\mathcal{H}$ insertion operation to one
  edge-quadruple in every equivalence class of eligible edge-quadruples.

\item Accept each graph with $n$ vertices if and only if the inserted
  $\mathcal{H}$ has minimal value of $(x_0,\dots ,x_6)$ among all
  possible $\mathcal{H}$ reductions of the graph.
\end{enumerate}

Then exactly one representative of each isomorphism
class of connected cubic graphs with $n$ vertices and girth at least $k$ is accepted.
\end{theorem}

\subsubsection{$\mathcal{H}$ operation for generating snarks}

Similar to the tripod operation, also the $\mathcal{H}$
operation allows a look-ahead which often allows to avoid the
generation of $3$-edge-colourable graphs with girth at least~$k$.

\begin{definition}
Let $\{\{e_1,e_2\},\{e_3,e_4\}\}$ be an edge-quadruple in a $3$-edge-coloured cubic graph
and $\{c_1,c_2,c_3\}= \{1,2,3\}$.

A $(c_1,c_2)$-path from $e_1$ to $e_3$ together with an edge-disjoint $(c_1,c_3)$-path from $e_2$ to $e_4$
is called a \underline{prune-pair} for $\{\{e_1,e_2\},\{e_3,e_4\}\}$.
\end{definition}

\begin{theorem}\label{theorem_h_snarks}
Given a cubic graph $G$ and a $3$-edge-colouring $c$ of $G$. 
If there is a prune-pair for an edge-quadruple $\{\{e_1,e_2\},\{e_3,e_4\}\}$,
then the graph obtained by extending $\{\{e_1,e_2\},\{e_3,e_4\}\}$ is $3$-edge-colourable.
\end{theorem}

\begin{proof}

Suppose we apply the $\mathcal{H}$ expansion operation to
$\{\{e_1,e_2\},\{e_3,e_4\}\}$ and call the obtained graph $G'$. We can
obtain a proper $3$-edge-colouring $c'$ of $G'$ by first starting from
the colouring $c$ of $G$ and performing a Kempe-chain-like colour
interchange of colours along the two paths of the prune-pair (see Figure~\ref{fig:recolouring_h}).

Finally we colour the edges of the inserted $\mathcal{H}$ which are
connected to $e_1$ and $e_3$ with the colour not present in the path
connecting them, and analogously with $e_2$ and $e_4$.
The central edge can then be coloured with the colour that occurred in both paths.
 The result is a proper
$3$-edge-colouring of $G'$.

\end{proof}

\begin{figure}[h!t]
    \centering
    \subfloat[]{
    \label{fig:recolour-h-a}
    \includegraphics[width=0.45\textwidth]{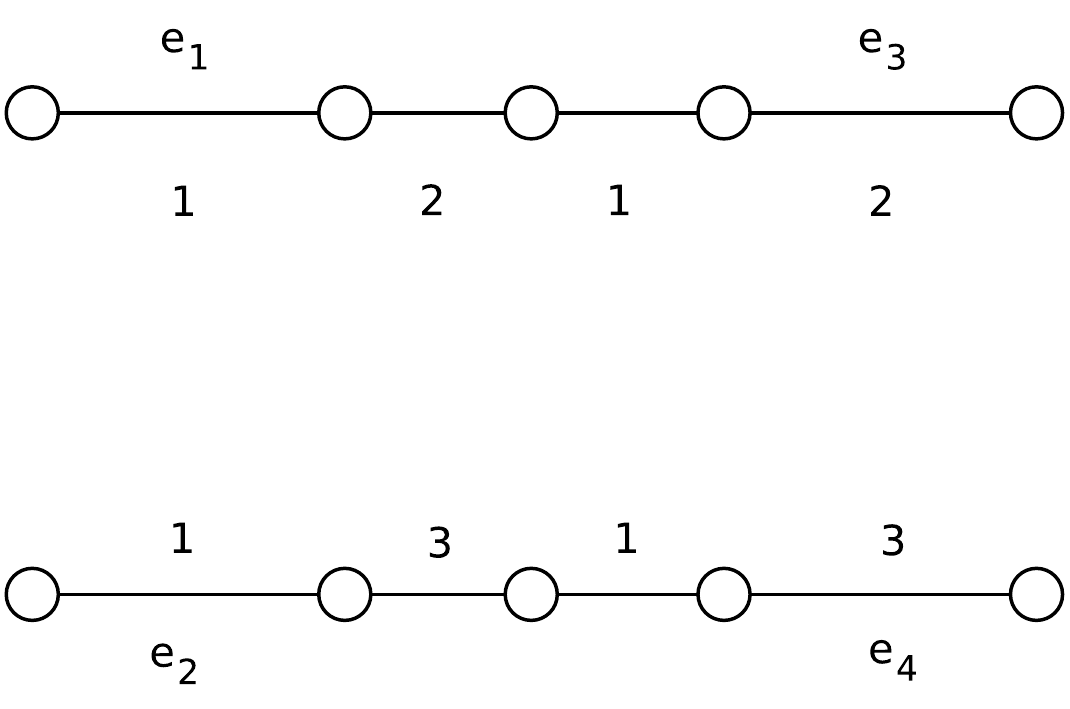}
    }\qquad
    \subfloat[]{
    \label{fig:recolour-h-b}
    \includegraphics[width=0.45\textwidth]{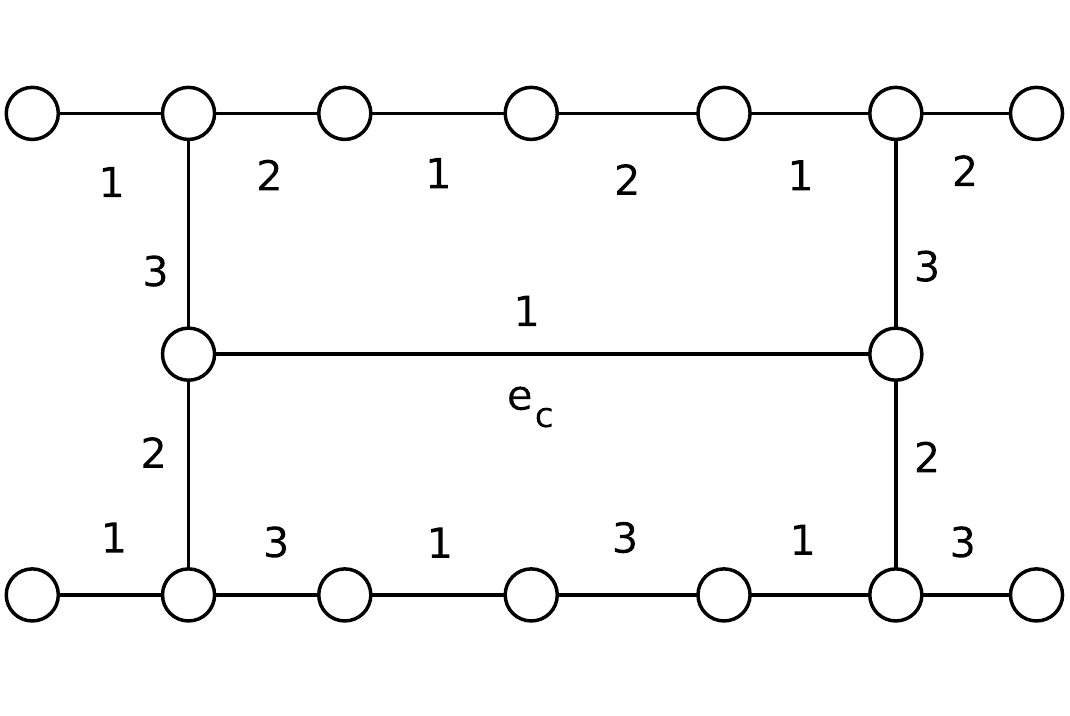}
    }    
    \caption{Procedure to recolour a graph after applying the $\mathcal{H}$ insertion operation to the edge-quadruple $\{\{e_1,e_2\},\{e_3,e_4\}\}$ used in the proof of Theorem~\ref{theorem_h_snarks}.}
    \label{fig:recolouring_h}
\end{figure}

We apply this look-ahead as follows for generating snarks with $n$
vertices and girth at least $k$. When a cubic graph $G$ with $n-6$
vertices is generated, we test if it is $3$-edge-colourable. If it is,
we remove all edge-quadruples which can be discarded due to
Theorem~\ref{theorem_h_snarks}. This allows to discard about $47$\% of
the edge-quadruples which would otherwise be expanded for $k=6$
and $n=36$. We then also compute up to five additional colourings (as
long as there are enough eligible edge-quadruples) and do the
same. These five additional colourings raise the total percentage of
discarded edge-quadruples to about $93$\% for $k=6$ and $n=36$.

The cost for computing a seventh colouring turned out to be higher
than the gain achieved by the additionally discarded edge-quadruples.

\section{Results}

\label{section:results}

\subsection{Number of cubic graphs and running times}

\label{subsect:running_times}

We implemented the tripod algorithm from Section~\ref{subsect:tripod}
for generating connected cubic graphs with girth at least $k$ for
$k=5,6$ and $7$ and the $\mathcal{H}$ algorithm from
Section~\ref{subsect:h-oper} for $k=6$ and $7$.

We will refer to our implementation of the tripod algorithm as
\textit{snarkhunter-tripod} and to our implementation of the
$\mathcal{H}$ algorithm as \textit{snarkhunter-H}, since we
incorporated it in the existing generator for cubic graphs
from~\cite{brinkmann_11} called \textit{snarkhunter}.

In Tables~\ref{table:times_cubic_g5}-\ref{table:times_cubic_g7} we
give the counts of connected cubic graphs with girth at least $k$ and
compare our new generators with the fastest existing generators for
cubic graphs, i.e.: \textit{minibaum} (which was developed by the
first author~\cite{brinkmann_96}), genreg (which was developed by
Meringer~\cite{meringer_99}) and \textit{snarkhunter} (which we
developed jointly with
McKay~\cite{brinkmann_11}). \textit{Snarkhunter} was only designed for
$k\in \{3,4,5\}$. To make the distinction between this generator and
\textit{snarkhunter-tripod} and \textit{snarkhunter-h}, we will refer
to it as \textit{snarkhunter-old}.

All running times reported in
Tables~\ref{table:times_cubic_g5}-\ref{table:times_cubic_g7} were
obtained by executing the generators on an Intel Xeon CPU E5-268 CPU at
2.50GHz and include writing the generated graphs to a null device. The
running times bigger than 3600 seconds include a small overhead due to
parallelisation.

As can be seen from Table~\ref{table:times_cubic_g5},
\textit{snarkhunter-tripod} is significantly faster than
\textit{minibaum} and \textit{genreg} for generating cubic graphs with
girth at least $5$, and is also faster -- though not by such a large factor --
than \textit{snarkhunter-old}.

As can be seen from Table~\ref{table:times_cubic_g6}, both
\textit{snarkhunter-tripod} and \textit{snarkhunter-h} are about
30 times faster than \textit{minibaum} and \textit{genreg} for
generating cubic graphs with girth at least $6$ and $34$ vertices and
the speedup seems to be increasing. For large orders,
\textit{snarkhunter-h} is a bit faster than
\textit{snarkhunter-tripod}.

As can be seen from Table~\ref{table:times_cubic_g7}, both
\textit{snarkhunter-tripod} and \textit{snarkhunter-h} are
a lot faster than \textit{minibaum} and \textit{genreg} for
generating cubic graphs with girth at least $7$ and the speedup seems
to be increasing very fast. For all orders which are within reach of
the algorithm, \textit{snarkhunter-h} is significantly faster than
\textit{snarkhunter-tripod}.

Our algorithms would also work for larger lower bounds $k$ on the
girth, but they are likely to become less efficient for practical
purposes for very large $k$. Note that for small vertex numbers
and $k\le 7$ the number of cubic graphs with
girth at least $k$ grows a lot faster than the number of cubic graphs
with girth at least $k-1$ (see
e.g.\ Tables~\ref{table:times_cubic_g5}-\ref{table:times_cubic_g7}). As
the algorithms described in this paper construct cubic graphs with
girth at least $k$ from smaller cubic graphs with girth at least
$k-1$, the running times of our generators grow a lot slower than the
running times of generators such as \textit{genreg} and
\textit{minibaum}. So for every $k$, the tripod algorithm or
$\mathcal{H}$ algorithm will most likely 
become faster than \textit{genreg} and \textit{minibaum} for
generating cubic graphs with girth at least $k$ from a certain number
of vertices on. However we estimate that already for $k=8$ the point
where our algorithms will become faster than \textit{genreg} and
\textit{minibaum} will most likely not be within the range which is
currently computationally feasible.

Tripod extension and $\mathcal{H}$ extension can be seen as the insertion of small trees.
For larger girth it might be useful to insert larger trees, but this will go together with a serious
increase of complexity of the algorithms.

In each case where a speedup is reported in
Tables~\ref{table:times_cubic_g5}-\ref{table:times_cubic_g7}, this
means that these graph counts have been independently confirmed by the
corresponding generators. Since all results were in complete
agreement, we believe that this provides strong evidence for the
correctness of our implementations and results.

\begin{table}
\centering
\small

	\begin{tabular}{|c || r | c | r | c | c | c |}
		\hline
		\multirow{2}{*}{Order} & \multirow{2}{*}{\# graphs} & \multirow{2}{*}{growth} & \multirow{2}{*}{sh-tripod (s)} & mb (s) / & gr (s) /  & sh-old (s) / \\
		 & & & & sh-tripod (s) & sh-tripod (s) & sh-tripod (s)  \\		
		\hline
22  &  90 938  &    &  0.7  &  8.14  &  7.33  &  1.57 \\
24  &  1 620 479  &  17.82  &  10.3  &  9.22  &  9.36  &  1.75 \\
26  &  31 478 584  &  19.43  &  195.3  &  9.26  &  10.44  &  1.67 \\
28  &  656 783 890  &  20.86  &  4 538  &  8.29  &  10.53  &  1.51 \\
30  &  14 621 871 204  &  22.26  &  117 048  &  7.43  &  9.53  &  1.43 \\
32  &  345 975 648 562  &  23.66  &  2 797 926  &    &    &  1.35 \\
		\hline
	\end{tabular}


\caption{Counts and generation times for connected cubic graphs with girth at least 5. \textit{Sh-tripod} stands for \textit{snarkhunter-tripod}, \textit{mb} for \textit{minibaum}, \textit{gr} for \textit{genreg} and \textit{sh-old} for \textit{snarkhunter-old}.}
\label{table:times_cubic_g5}

\end{table}

\begin{table}
\centering
\small

	\begin{tabular}{|c || r | c | r | c | c | c |}
		\hline
		\multirow{2}{*}{Order} & \multirow{2}{*}{\# graphs} & \multirow{2}{*}{growth} & \multirow{2}{*}{sh-h (s)} & mb (s) / & gr (s) /  & sh-tripod (s) / \\
		 & & & & sh-h (s) & sh-h (s) & sh-h (s)  \\		
		\hline
24  &  7 574  &    &  0.2  &  16.00  &  10.30  &  1.00 \\
26  &  181 227  &  23.93  &  3.4  &  20.21  &  13.41  &  1.00 \\
28  &  4 624 501  &  25.52  &  64  &  25.30  &  17.17  &  1.03 \\
30  &  122 090 544  &  26.40  &  1 528.5  &  31.53  &  18.71  &  1.06 \\
32  &  3 328 929 954  &  27.27  &  36 579  &  33.21  &  22.86  &  1.08 \\
34  &  93 990 692 595  &  28.23  &  862 375  &  34.86  &  31.78  &  1.18 \\
36  &  2 754 222 605 376  &  29.30  &  22 403 689  &    &    &  1.30 \\
		\hline
	\end{tabular}


\caption{Counts and generation times for connected cubic graphs with girth at least 6. \textit{Sh-h} stands for \textit{snarkhunter-h}, \textit{sh-tripod} stands for \textit{snarkhunter-tripod}, \textit{mb} for \textit{minibaum} and \textit{gr} for \textit{genreg}.}
\label{table:times_cubic_g6}

\end{table}

\begin{table}
\centering
\small

	\begin{tabular}{|c || r | c | r | c | c | c |}
		\hline
		\multirow{2}{*}{Order} & \multirow{2}{*}{\# graphs} & \multirow{2}{*}{growth} & \multirow{2}{*}{sh-h (s)} & mb (s) / & gr (s) /  & sh-tripod (s) / \\
		 & & & & sh-h (s) & sh-h (s) & sh-h (s)  \\		
		\hline
30  &  546  &    &  0.4  &  17.00  &  7.03  &  11.00 \\
32  &  30 368  &  55.62  &  7  &  38.61  &  10.48  &  11.03 \\
34  &  1 782 840  &  58.71  &  186.1  &  83.68  &  20.25  &  9.02 \\
36  &  95 079 083  &  53.33  &  6 968  &  126.95  &  26.94  &  7.41 \\
38  &  4 686 063 120  &  49.29  &  214 412  &  &  43.76  &  5.33 \\
40  &  220 323 447 962  &  47.02  &  7 049 445  &    &    &  4.72 \\
42  &  10 090 653 722 861  &  45.80  &  383 939 689  &    &    &   \\
		\hline
	\end{tabular}


\caption{Counts and generation times for connected cubic graphs with girth at least $7$. \textit{Sh-h} stands for \textit{snarkhunter-h}, \textit{sh-tripod} stands for \textit{snarkhunter-tripod}, \textit{mb} for \textit{minibaum} and \textit{gr} for \textit{genreg}.}
\label{table:times_cubic_g7}

\end{table}

We incorporated the tripod and $\mathcal{H}$ algorithm into
\textit{snarkhunter-old} such that this program can now be used to
generate all connected cubic graphs with girth at least $k$ for $3 \le
k \le 7$. The latest version of this generator can be downloaded
from~\cite{snarkhunter-site}. Several of the lists of cubic graphs from Tables~\ref{table:times_cubic_g5}-\ref{table:times_cubic_g7} can be
downloaded from the \textit{House of Graphs}~\cite{hog} at
\url{http://hog.grinvin.org/Cubic}~.

\subsection{New lists of snarks and their properties}

\label{subsect:results_snarks}

\subsubsection{Numbers of snarks}

We also incorporated the look-aheads described in
Section~\ref{section:generation_algorithm} into the tripod and
$\mathcal{H}$ algorithm so that snarks with girth at least $k$ can be
generated more efficiently. This is significantly faster than just
generating all cubic graphs with girth at least $k$ and applying a
filter for colourability at the end (which was for
$k > 5$ up until now the only
way to generate complete lists of snarks with girth at least $k$).

More specifically, for the tripod algorithm this is about $2$ times
faster than the filter approach for $k=5$ and about $1.5$ times faster
than the filter approach for $k=6$. For the $\mathcal{H}$ algorithm
this is about $2$ times faster than the filter approach for $k=6$ and about
$20$\% faster than the filter approach for $k=7$. For higher girth there is more time required for generating a graph, so here the ratio for testing if a graph is colourable is smaller.

The numbers of snarks are shown in
Table~\ref{table:number_of_snarks}. Previously the complete list of
all snarks up to $36$ vertices was known and it was also known that
there are no snarks with girth at least $7$ up to at least 38 vertices.

Using the new algorithms we were able to generate all snarks with
girth at least $6$ up to $38$ vertices. This took about $15$ CPU
years. 

As already mentioned in Section~\ref{section:intro}, the smallest
snark with girth $7$ is currently unknown. Using the new algorithms we
generated all snarks with girth at least $7$ up to $42$ vertices
(which took about $12$ CPU years). As this did not yield any snarks,
this shows that the smallest snark of girth $7$ must have at least $44$
vertices.

The tripod algorithm is faster than \textit{snarkhunter-old} for
generating snarks (i.e.\ girth at least $5$), but as the number of
snarks grows very fast, it is still computationally infeasible to
generate all snarks with $38$ vertices. However we did generate a
sample of snarks with $38$ vertices and a sample of snarks with girth
at least $6$ with $40$ vertices. These numbers are indicated with a
'$\ge$' in Table~\ref{table:number_of_snarks}. We estimate that it
would take about $450$ CPU years to generate all snarks with girth at
least $6$ with $40$ vertices on the computers used here.

We used the tripod algorithm to generate all snarks up to $34$ vertices
and this confirmed the results in~\cite{snark-paper}.

The complete list of snarks with girth at least $6$ and $38$ vertices
were independently obtained by the tripod and $\mathcal{H}$
algorithm. However, as there are only $39$ such graphs, this does not
provide very strong evidence for the correctness of these
programs. Therefore we actually generated all connected cubic graphs
with girth at least $6$ which are not $3$-edge-colourable (so not just
the cyclically $4$-edge-connected ones). Both algorithms yielded the
same $26~365$ graphs.

Also both the tripod algorithm and the $\mathcal{H}$ algorithm
independently confirmed that there are no snarks with girth at least
$7$ up to $40$ vertices. The computation for $42$ vertices was only
performed using the $\mathcal{H}$ algorithm. Since all results were in
complete agreement, we believe that this provides strong evidence for
the correctness of our implementations and results.

All snarks from Table~\ref{table:number_of_snarks} can be downloaded
from the \textit{House of Graphs}~\cite{hog} at \url{http://hog.grinvin.org/Snarks}~.
The adjacency lists from the
snarks with $38$ vertices and girth $6$  can be inspected in the database
of interesting graphs at the \textit{House of Graphs} by
searching for the keywords ``snark * girth 6''.

\begin{table}
\centering
\small

	\begin{tabular}{|c || r | c | c | r | c | c |}
		\hline
		Order & \# snarks & $g \geq 6$ & $g \geq 7$ & $\lambda_c \geq 5$ &  $\lambda_c \geq 5$ and $g \geq 6$ & $\lambda_c \geq 6$\\
		\hline
		10 	& 1     & 0 				& 0 		& 1 & 0 & 0 \\
		12,14,16 & 0 	  & 0		         & 0		& 0 & 0 & 0 \\
18  &  2  &  0  &  0  &  0  &  0  &  0 \\
20  &  6  &  0  &  0  &  1  &  0  &  0 \\
22  &  20  &  0  &  0  &  2  &  0  &  0 \\
24  &  38  &  0  &  0  &  2  &  0  &  0 \\
26  &  280  &  0  &  0  &  10  &  0  &  0 \\
28  &  2 900  &  1  &  0  &  75  &  1  &  1 \\
30  &  28 399  &  1  &  0  &  509  &  1  &  0 \\
32  &  293 059  &  0  &  0  &  2 953  &  0  &  0 \\
34  &  3 833 587  &  0  &  0  &  19 935  &  0  &  0 \\
36  &  60 167 732  &  1  &  0  &  180 612  &  1  &  1 \\
38  &  $\geq$ 19 775 768  &  39  &  0  &  $\geq$ 35 429  &  0  &  0 \\
40  &  ?  &  $\geq 25$  &  0  &  ?  &  $\geq 0$  &  ? \\
42  &  ?  &  ?  &  0  &  ?  &  ?  &  ? \\
		
		\hline
	\end{tabular}


\caption{The number of snarks. The columns with a header of the form $g \geq a$ contain the number of snarks with girth at least $a$ and those of the form $\lambda_c \geq b$ contain the number of snarks with cyclic edge-connectivity at least $b$.}
\label{table:number_of_snarks}

\end{table}

\subsubsection{Properties of the new snarks}

In this section we present an analysis of the new lists of snarks. None of the snarks with girth at least $6$ and $38$ vertices is
cyclically $5$-edge-connected.

The \textit{dot product operation} which was originally introduced by Isaacs in~\cite{isaacs_75} allows to construct infinitely many snarks. When this operation is applied to two snarks $G_1$ and $G_2$, the result is a snark with $|V(G_1)| + |V(G_2)| - 2$ vertices.

We implemented a program which performs the dot product operation and applied it to the complete lists of snarks. This resulted in the following observations.

\begin{observation}
All of the 39 snarks with girth 6 and 38 vertices can be obtained by applying the dot product to two copies of the flower snark $J_5$.
\end{observation}

\begin{observation}
There are no snarks with girth at least 6 and 40 or 42 vertices which can be obtained by applying the dot product to a pair of snarks.
\end{observation}

\begin{observation}\label{obs:dotproduct44}
There are 4 snarks with girth at least 6 and 44 vertices which can be obtained by applying the dot product to a pair of snarks.
\end{observation}

For 46 vertices, we applied the dot product to every pair of snarks $\{G_1,G_2\}$ which results in a snark on 46 vertices, except for the pairs where a snark with 38 vertices is combined with the Petersen graph as not all snarks with 38 vertices are known.

\begin{observation}\label{obs:dotproduct46}
There are at least 876 snarks with girth at least 6 and 46 vertices which can be obtained by applying the dot product to a pair of snarks.
\end{observation}

All of the snarks from Observation~\ref{obs:dotproduct44} and~\ref{obs:dotproduct46} were obtained by applying the dot product operation to pairs of snarks where the flower snark $J_5$ is one of the two snarks.

A graph $G$ is \textit{hypohamiltonian} if $G$ is not hamiltonian but $G-v$ is hamiltonian for every vertex $v$ in $G$.

\begin{observation}
There are exactly 29 hypohamiltonian snarks with girth at least 6 and 38 vertices.
\end{observation}

In the remainder we report how we used our new lists of snarks to test several conjectures, but it did not yield any counterexamples. For more information about the conjectures, we refer to~\cite{snark-paper}. We tested these conjectures using independently tested programs which were already used in~\cite{snark-paper}.

A \emph{dominating} cycle is a cycle $C$ such that every edge in $G$ has an endpoint on $C$. Fleischner~\cite{fleischner1988some} made the following conjecture related to dominating cycles:

\begin{conjecture}[Fleischner~\cite{fleischner1988some}]\label{conj:doms}
Every snark has a dominating cycle. 
\end{conjecture}

\begin{observation}
Conjecture~\ref{conj:doms} holds for all snarks with girth 6 with at most 38 vertices.
\end{observation}

Recall that a \textit{$k$-factor} of a graph $G$ is a spanning $k$-regular subgraph of $G$. The \emph{oddness} of a bridgeless cubic graph is the minimum number of odd order  components in any 2-factor of the graph. The oddness of a cubic graph provides a measure for how far a the graph is from being colourable, as a graph has oddness 0 if and only if it is colourable.

Together with H{\"a}gglund and Markstr{\"o}m we determined in~\cite{snark-paper} that all snarks with at most 36 vertices have oddness 2 and asked for the smallest snark with oddness greater than 2 (i.e. Problem~2 in~\cite{snark-paper}). Recently Lukot'ka et al.~\cite{lukotka2015small} have constructed a snark with oddness 4 and 44 vertices and they conjectured that it is the smallest snark with oddness~4.

\begin{observation}
All snarks with girth at least 6 and at most 38 vertices have oddness 2.
\end{observation}

Next to oddness, another measure for how far a graph is from being colourable is the notion of \textit{strong snarks} which was introduced by Jaeger in~\cite{jaeger1985survey}. A snark $G$ is \textit{strong} if for every edge $e$ of $G$ the application of the edge reduction operation from Figure~\ref{fig:edgeinsert_old} results in an uncolourable graph. Celmins~\cite{celmins1985cubic} proved that a minimum counterexample to the cycle double cover conjecture must be a strong snark.

In~\cite{snark-paper} we determined with H{\"a}gglund and Markstr{\"o}m that there are no strong snarks with less than 34 vertices, exactly 7 strong snarks with 34 vertices and 25 strong snarks with 36 vertices.

\begin{observation}
None of the snarks with girth at least 6 and at most 38 vertices is strong.
\end{observation}

We also applied the edge insertion operation from Figure~\ref{fig:edgeinsert_old} in all possible ways to the complete list of snarks with 36 vertices and then tested which of the resulting graphs were strong snarks. This yielded the following result.

\begin{observation}
There are at least 298 strong snarks with 38 vertices.
\end{observation}

One might speculate that this is the complete set of strong snarks with 38 vertices, since applying the same procedure to all snarks with 32 and 34 vertices yields all 7 and 25 snarks with 34 and 36 vertices, respectively.

These strong snarks can also be obtained from the database of interesting graphs at the \textit{House of Graphs} by searching for the keywords ``strong snark''.

The \emph{total chromatic number} of a graph $G$ is the minimum number of colours required to colour $E(G)\cup V(G)$ such that adjacent vertices and edges have different colours and no vertex has the same colour as its incident edges. It is known~\cite{rosenfeld1971total} that for cubic graphs the total chromatic number is either 4 or 5.

Cavicchioli et al.\ asked for the smallest snark with total chromatic number 5 (if any) (i.e.\ Problem 5.1 in~\cite{cavicchioli2003special}). Jointly with H{\"a}gglund and Markstr{\"o}m, we showed in~\cite{snark-paper} that every snark up to at least 36 vertices has total chromatic number 4. 

Recently, Brinkmann, Preissmann and Sasaki~\cite{brinkmann2015snarks} gave a construction which yields \textit{weak snarks} with total chromatic number 5 for every even order $n \ge 40$. A \textit{weak snark} is a cyclically 4-edge-connected cubic graph with chromatic index 4 and girth at least 4.

\begin{observation}
All snarks with girth at least 6 and at most 38 vertices have total chromatic number 4.
\end{observation}

In~\cite{jaeger1988nowhere} Jaeger made the following conjecture which is known as the \textit{Petersen colouring conjecture} (we refer to~\cite{jaeger1988nowhere} for more details).

\begin{conjecture}[Jaeger~\cite{jaeger1988nowhere}]\label{conj:petersen_col}
If $G$ is a bridgeless cubic graph, one can colour the edges of $G$ using the edges of the Petersen graph as colours in such a way that any three mutually adjacent edges of $G$ are coloured by three edges that are mutually adjacent in the Petersen graph.
\end{conjecture}

\begin{observation}
Conjecture~\ref{conj:petersen_col} holds for all snarks with girth 6 with at most 38 vertices.
\end{observation}

We also verified that there are no counterexamples to the above mentioned conjectures among the samples of snarks with 38 vertices and the samples of snarks with girth 6 and 40 vertices which are indicated with a '$\ge$' in Table~\ref{table:number_of_snarks}.

\section*{Acknowledgements}
Most computations for this work were carried out using the Stevin Supercomputer Infrastructure at Ghent University.

\section*{References}

\bibliographystyle{plain}
\bibliography{references}

\begin{thebibliography}{10}

\bibitem{balaban_66}
A.T. Balaban.
\newblock Valence-isomerism of cyclopolyenes.
\newblock {\em Revue Roumaine de chimie}, 11(9):1097--1116, 1966.

\bibitem{brinkmann_96}
G.~Brinkmann.
\newblock Fast generation of cubic graphs.
\newblock {\em Journal of Graph Theory}, 23(2):139--149, 1996.

\bibitem{hog}
G.~Brinkmann, K.~Coolsaet, J.~Goedgebeur, and H.~M{\'e}lot.
\newblock {House of Graphs: a database of interesting graphs}.
\newblock {\em Discrete Applied Mathematics}, 161(1-2):311--314, 2013.
\newblock Available at \url{http://hog.grinvin.org/}.

\bibitem{snarkhunter-site}
G.~Brinkmann and J.~Goedgebeur.
\newblock Homepage of snarkhunter: \url{http://caagt.ugent.be/cubic/}.

\bibitem{snark-paper}
G.~Brinkmann, J.~Goedgebeur, J.~H{\"a}gglund, and K.~Markstr{\"o}m.
\newblock Generation and properties of snarks.
\newblock {\em Journal of Combinatorial Theory, Series B}, 103(4):468--488,
  2013.

\bibitem{brinkmann_11}
G.~Brinkmann, J.~Goedgebeur, and B.D. McKay.
\newblock Generation of cubic graphs.
\newblock {\em Discrete Mathematics and Theoretical Computer Science},
  13(2):69--80, 2011.

\bibitem{brinkmann2013history}
G.~Brinkmann, J.~Goedgebeur, and N.~Van~Cleemput.
\newblock The history of the generation of cubic graphs.
\newblock {\em International Journal of Chemical Modeling}, 5:67--89, 2013.

\bibitem{brinkmann2015snarks}
G.~Brinkmann, M.~Preissmann, and D.~Sasaki.
\newblock Snarks with total chromatic number 5.
\newblock {\em Discrete Mathematics \& Theoretical Computer Science},
  17(1):369--382, 2015.

\bibitem{cavicchioli2003special}
A.~Cavicchioli, T.E. Murgolo, B.~Ruini, and F.~Spaggiari.
\newblock Special classes of snarks.
\newblock {\em Acta Applicandae Mathematica}, 76(1):57--88, 2003.

\bibitem{celmins1985cubic}
U.~Celmins.
\newblock Cubic graphs that do not have an edge 3-colouring.
\newblock {\em Dissertation Abstracts International Part B: Science and
  Engineering}, 46(5), 1985.

\bibitem{de_vries_1889}
J.~de~Vries.
\newblock Over vlakke configuraties waarin elk punt met twee lijnen incident
  is.
\newblock {\em Mededeelingen der Koninklijke Akademie voor Wetenschappen,
  Afdeeling Natuurkunde}, 3(6):382--407, 1889.

\bibitem{de_vries_1891}
J.~de~Vries.
\newblock Sur les configurations planes dont chaque point supporte deux
  droites.
\newblock {\em Rendiconti Circolo Matematico Palermo}, 5:221--226, 1891.

\bibitem{fleischner1988some}
H.~Fleischner.
\newblock Some blood, sweat, but no tears in {E}ulerian graph theory.
\newblock {\em Congressus Numerantium}, 63:8--48, 1988.

\bibitem{huck2000reducible}
A.~Huck.
\newblock Reducible configurations for the cycle double cover conjecture.
\newblock {\em Discrete Applied Mathematics}, 99(1):71--90, 2000.

\bibitem{isaacs_75}
R.~Isaacs.
\newblock {Infinite families of nontrivial trivalent graphs which are not Tait
  colorable}.
\newblock {\em American Mathematical Monthly}, 82(3):221--239, 1975.

\bibitem{jaeger1985survey}
F.~Jaeger.
\newblock A survey of the cycle double cover conjecture.
\newblock {\em North-Holland Mathematics Studies}, 115:1--12, 1985.

\bibitem{jaeger1988nowhere}
F.~Jaeger.
\newblock Nowhere-zero flow problems.
\newblock {\em Selected topics in graph theory}, 3:71--95, 1988.

\bibitem{JS1980}
F.~Jaeger and T.~Swart.
\newblock Conjecture 1.
\newblock In M.~Deza and I.~G. Rosenberg, editors, {\em Combinatorics 79,
  {P}art {II}, Annals of Discrete Mathematics}, volume~9, page 305, 1980.

\bibitem{kochol1996snarks}
M.~Kochol.
\newblock Snarks without small cycles.
\newblock {\em Journal of Combinatorial Theory, Series B}, 67(1):34--47, 1996.

\bibitem{kroto_85}
H.W. Kroto, J.R. Heath, S.C. O'Brien, R.F. Curl, and R.E. Smalley.
\newblock {$C_{60}$}: Buckminsterfullerene.
\newblock {\em Nature}, 318(6042):162--163, 1985.

\bibitem{lukotka2015small}
R.~Lukot'ka, E.~M{\'a}{\v{c}}ajov{\'a}, J.~Maz{\'a}k, and M.~{\v{S}}koviera.
\newblock Small snarks with large oddness.
\newblock {\em The Electronic Journal of Combinatorics}, 22(1):20, 2015.

\bibitem{nauty-website}
B.D. McKay.
\newblock {nauty User's Guide (Version 2.5)}.
\newblock Technical Report TR-CS-90-02, Department of Computer Science,
  Australian National University. The latest version of the software is
  available at \url{http://cs.anu.edu.au/~bdm/nauty}.

\bibitem{mckay_98}
B.D. McKay.
\newblock Isomorph-free exhaustive generation.
\newblock {\em Journal of Algorithms}, 26(2):306--324, 1998.

\bibitem{mckay1998fast}
B.D. McKay, W.~Myrvold, and J.~Nadon.
\newblock Fast backtracking principles applied to find new cages.
\newblock {\em 9th Annual ACM-SIAM Symposium on Discrete Algorithms}, pages
  188--191, 1998.

\bibitem{mckay_14}
B.D. McKay and A.~Piperno.
\newblock Practical graph isomorphism, {II}.
\newblock {\em Journal of Symbolic Computation}, 60:94--112, 2014.

\bibitem{meringer_99}
M.~Meringer.
\newblock Fast generation of regular graphs and construction of cages.
\newblock {\em Journal of Graph Theory}, 30(2):137--146, 1999.

\bibitem{robinson_83}
R.W. Robinson and N.C. Wormald.
\newblock Numbers of cubic graphs.
\newblock {\em Journal of Graph Theory}, 7:463--467, 1983.

\bibitem{robinson_92}
R.W. Robinson and N.C. Wormald.
\newblock Almost all cubic graphs are hamiltonian.
\newblock {\em Random Structures \& Algorithms}, 3(2):117--125, 1992.

\bibitem{rosenfeld1971total}
M.~Rosenfeld.
\newblock On the total coloring of certain graphs.
\newblock {\em Israel Journal of Mathematics}, 9(3):396--402, 1971.

\end{thebibliography}

\section*{Appendix: Adjacency lists}

This section contains the adjacency lists of the complete list of 39 snarks with girth 6 and 38 vertices. These graphs can also be found on the \textit{House of Graphs}~\cite{hog} at \url{http://hog.grinvin.org/Snarks}~.

The graphs are represented in the following compact format. The vertices of a graph $G$ are numbered from $0$ to $|V(G)| - 1$ and each entry corresponds to the adjacency list of a graph where for a given vertex $v$, only the neighbours of $v$ which have a larger label than $v$ are listed. So the first 3 numbers are the neighbours of vertex 0, followed by the neighbours of vertex 1 which have a higher label than 1, etc.

\begin{enumerate}

 \item \{4, 8, 30, 6, 9, 14, 4, 7, 31, 7, 8, 15, 5, 19, 20, 7, 10, 9, 11, 13, 30, 13, 31, 13, 14, 16, 15, 17, 18, 34, 28, 32, 22, 24, 23, 26, 25, 27, 23, 25, 32, 23, 35, 25, 36, 29, 36, 29, 35, 29, 34, 33, 33, 33, 37, 37, 37\}

 \item \{4, 8, 30, 6, 9, 14, 4, 7, 31, 7, 8, 15, 5, 19, 20, 7, 10, 9, 11, 13, 30, 13, 31, 13, 14, 16, 15, 17, 18, 28, 28, 34, 22, 24, 26, 35, 27, 36, 23, 25, 32, 23, 27, 35, 25, 26, 36, 29, 29, 29, 33, 33, 33, 34, 37, 37, 37\}

 \item \{4, 8, 16, 6, 9, 14, 4, 7, 18, 7, 8, 15, 20, 15, 22, 28, 7, 10, 9, 11, 13, 16, 13, 18, 13, 14, 21, 15, 17, 21, 25, 19, 24, 30, 21, 26, 32, 34, 29, 32, 35, 28, 33, 27, 29, 31, 34, 31, 36, 29, 31, 35, 33, 36, 37, 37, 37\}

 \item \{4, 8, 20, 5, 9, 11, 7, 9, 12, 7, 17, 18, 5, 12, 17, 7, 11, 15, 9, 15, 16, 19, 22, 19, 13, 26, 34, 15, 18, 21, 17, 20, 19, 28, 23, 34, 24, 27, 30, 32, 25, 31, 28, 32, 31, 33, 29, 33, 29, 30, 35, 35, 36, 36, 37, 37, 37\}

 \item \{4, 8, 10, 5, 9, 11, 7, 9, 22, 5, 7, 18, 5, 12, 7, 11, 20, 9, 15, 11, 13, 24, 28, 14, 19, 15, 17, 20, 27, 29, 32, 18, 27, 21, 21, 25, 21, 23, 29, 33, 34, 30, 35, 30, 34, 28, 33, 36, 31, 29, 31, 36, 33, 35, 37, 37, 37\}

 \item \{4, 8, 10, 5, 9, 11, 7, 9, 22, 5, 7, 18, 5, 12, 7, 11, 20, 9, 15, 11, 13, 24, 28, 14, 19, 15, 17, 20, 17, 29, 34, 18, 21, 21, 25, 21, 26, 34, 25, 28, 30, 27, 31, 27, 30, 32, 32, 33, 31, 33, 35, 35, 36, 36, 37, 37, 37\}

 \item \{4, 8, 10, 5, 9, 11, 7, 9, 22, 5, 7, 18, 5, 12, 7, 11, 20, 9, 15, 11, 13, 16, 28, 14, 19, 15, 17, 20, 32, 34, 18, 26, 21, 21, 25, 21, 23, 30, 28, 33, 29, 32, 35, 27, 29, 31, 34, 31, 36, 29, 31, 35, 33, 36, 37, 37, 37\}

 \item \{4, 8, 10, 5, 9, 11, 7, 9, 22, 5, 7, 18, 5, 12, 7, 11, 20, 9, 15, 11, 13, 28, 34, 14, 19, 15, 17, 20, 24, 26, 32, 18, 24, 21, 21, 23, 21, 25, 26, 27, 34, 29, 29, 33, 30, 30, 33, 31, 32, 31, 35, 35, 36, 36, 37, 37, 37\}

 \item \{8, 10, 18, 6, 9, 12, 7, 11, 32, 5, 7, 8, 17, 20, 32, 17, 34, 10, 35, 35, 36, 11, 36, 13, 13, 15, 34, 15, 15, 24, 26, 21, 23, 28, 23, 19, 30, 20, 29, 21, 26, 23, 31, 33, 25, 33, 28, 30, 27, 29, 31, 29, 31, 33, 37, 37, 37\}

 \item \{4, 8, 10, 5, 9, 11, 7, 9, 16, 5, 7, 18, 5, 12, 7, 11, 20, 9, 15, 11, 13, 28, 34, 14, 19, 15, 17, 20, 26, 34, 18, 25, 21, 21, 22, 21, 23, 27, 29, 30, 25, 27, 31, 29, 30, 32, 32, 31, 33, 33, 35, 35, 36, 36, 37, 37, 37\}

 \item \{4, 8, 14, 5, 9, 20, 4, 11, 15, 7, 8, 12, 21, 12, 22, 7, 10, 19, 15, 9, 11, 13, 14, 13, 13, 17, 17, 17, 19, 20, 24, 28, 34, 24, 27, 28, 35, 26, 29, 27, 33, 35, 30, 30, 33, 36, 31, 32, 31, 29, 36, 31, 33, 34, 37, 37, 37\}

 \item \{4, 8, 14, 5, 9, 20, 4, 11, 15, 7, 8, 12, 21, 12, 22, 7, 10, 19, 15, 9, 11, 13, 14, 13, 13, 17, 17, 17, 19, 20, 21, 30, 34, 24, 27, 28, 25, 29, 28, 33, 35, 30, 36, 33, 36, 27, 29, 32, 31, 29, 31, 35, 33, 34, 37, 37, 37\}

 \item \{4, 20, 28, 6, 9, 29, 12, 30, 34, 7, 8, 21, 5, 11, 21, 29, 7, 20, 26, 9, 28, 30, 13, 14, 26, 15, 34, 16, 18, 18, 22, 16, 23, 17, 19, 35, 22, 35, 36, 23, 36, 24, 24, 25, 25, 27, 32, 27, 32, 31, 31, 33, 33, 33, 37, 37, 37\}

 \item \{4, 8, 14, 6, 9, 16, 4, 7, 15, 7, 8, 17, 32, 22, 26, 34, 7, 10, 9, 11, 13, 14, 13, 15, 13, 16, 33, 30, 30, 17, 19, 25, 33, 35, 24, 27, 21, 26, 32, 25, 36, 23, 28, 27, 36, 25, 34, 29, 29, 31, 35, 31, 31, 33, 37, 37, 37\}

 \item \{4, 8, 14, 6, 9, 16, 4, 7, 15, 7, 8, 17, 20, 22, 24, 34, 7, 10, 9, 11, 13, 14, 13, 15, 13, 16, 18, 30, 30, 17, 32, 25, 28, 24, 27, 32, 21, 26, 23, 25, 23, 28, 35, 25, 29, 34, 29, 35, 36, 31, 33, 33, 36, 33, 37, 37, 37\}

 \item \{4, 8, 20, 5, 9, 11, 7, 9, 12, 7, 17, 18, 5, 12, 17, 7, 11, 15, 9, 15, 16, 19, 24, 19, 21, 24, 28, 32, 15, 18, 22, 17, 20, 19, 23, 26, 34, 25, 28, 29, 34, 27, 27, 33, 30, 32, 30, 31, 31, 33, 35, 35, 36, 36, 37, 37, 37\}

 \item \{4, 8, 20, 5, 9, 11, 7, 9, 12, 7, 17, 18, 5, 12, 17, 7, 11, 15, 9, 15, 13, 16, 19, 19, 24, 26, 28, 15, 18, 22, 17, 20, 19, 27, 28, 32, 34, 23, 29, 32, 35, 25, 30, 29, 33, 31, 36, 31, 35, 29, 31, 34, 33, 36, 37, 37, 37\}

 \item \{4, 8, 20, 5, 9, 11, 7, 9, 12, 7, 17, 18, 5, 12, 17, 7, 11, 15, 9, 15, 13, 16, 19, 19, 24, 30, 34, 15, 18, 22, 17, 20, 19, 27, 24, 32, 34, 25, 31, 28, 32, 35, 29, 29, 33, 27, 30, 36, 28, 29, 31, 35, 33, 36, 37, 37, 37\}

 \item \{4, 10, 30, 6, 9, 14, 4, 11, 31, 7, 8, 15, 18, 18, 22, 26, 7, 10, 31, 9, 30, 11, 13, 13, 13, 14, 16, 15, 17, 19, 20, 19, 22, 21, 28, 23, 24, 25, 34, 23, 27, 25, 35, 28, 29, 35, 29, 34, 36, 32, 33, 33, 33, 36, 37, 37, 37\}

 \item \{4, 10, 30, 6, 9, 14, 4, 11, 31, 7, 8, 15, 18, 18, 24, 32, 7, 10, 31, 9, 30, 11, 13, 13, 13, 14, 16, 15, 17, 19, 20, 19, 22, 26, 25, 23, 34, 23, 25, 35, 23, 36, 25, 27, 29, 35, 29, 34, 29, 32, 36, 33, 33, 33, 37, 37, 37\}

 \item \{4, 8, 16, 6, 9, 14, 4, 7, 32, 7, 8, 15, 5, 15, 18, 7, 10, 9, 11, 13, 16, 13, 32, 13, 14, 20, 15, 17, 24, 26, 23, 27, 30, 33, 34, 21, 33, 22, 31, 26, 28, 30, 35, 29, 34, 27, 29, 36, 27, 29, 35, 31, 36, 33, 37, 37, 37\}

 \item \{4, 8, 16, 6, 9, 14, 4, 7, 32, 7, 8, 15, 5, 15, 18, 7, 10, 9, 11, 13, 16, 13, 32, 13, 14, 20, 15, 17, 19, 24, 19, 22, 26, 21, 28, 27, 33, 25, 34, 25, 26, 35, 25, 36, 27, 29, 31, 35, 31, 34, 31, 33, 36, 33, 37, 37, 37\}

 \item \{4, 8, 16, 6, 9, 14, 4, 7, 32, 7, 8, 15, 5, 15, 18, 7, 10, 9, 11, 13, 16, 13, 32, 13, 14, 20, 15, 17, 19, 24, 19, 22, 23, 21, 30, 28, 34, 25, 34, 26, 29, 25, 35, 29, 27, 30, 28, 35, 31, 36, 33, 33, 36, 33, 37, 37, 37\}

 \item \{4, 8, 16, 6, 9, 14, 4, 7, 32, 7, 8, 15, 5, 15, 18, 7, 10, 9, 11, 13, 16, 13, 32, 13, 14, 19, 15, 17, 19, 20, 22, 28, 23, 21, 24, 26, 28, 27, 34, 25, 34, 25, 35, 30, 27, 36, 29, 31, 31, 35, 33, 36, 33, 33, 37, 37, 37\}

 \item \{4, 10, 34, 6, 9, 32, 4, 11, 35, 5, 7, 8, 15, 18, 36, 7, 10, 35, 9, 34, 11, 13, 13, 13, 17, 28, 15, 16, 22, 36, 21, 32, 23, 33, 19, 26, 20, 22, 21, 24, 23, 25, 29, 27, 28, 27, 29, 27, 30, 31, 31, 31, 33, 33, 37, 37, 37\}

 \item \{4, 10, 14, 6, 9, 18, 4, 11, 15, 7, 8, 19, 24, 22, 24, 32, 7, 10, 15, 9, 14, 11, 13, 13, 13, 18, 20, 17, 17, 17, 28, 32, 19, 23, 30, 34, 25, 26, 34, 27, 29, 25, 27, 25, 28, 31, 31, 35, 30, 35, 33, 36, 33, 36, 37, 37, 37\}

 \item \{4, 10, 14, 6, 9, 18, 4, 11, 15, 7, 8, 19, 24, 16, 22, 24, 7, 10, 15, 9, 14, 11, 13, 13, 13, 18, 20, 17, 17, 17, 26, 19, 23, 30, 34, 25, 26, 28, 27, 34, 29, 32, 25, 32, 31, 29, 31, 30, 35, 35, 33, 36, 33, 36, 37, 37, 37\}

 \item \{4, 10, 14, 6, 9, 18, 4, 11, 15, 7, 8, 19, 5, 16, 24, 7, 10, 15, 9, 14, 11, 13, 13, 13, 18, 20, 17, 17, 17, 26, 19, 23, 25, 28, 29, 32, 34, 23, 28, 33, 30, 31, 35, 32, 35, 27, 29, 31, 36, 29, 31, 34, 33, 36, 37, 37, 37\}

 \item \{4, 10, 14, 6, 9, 18, 4, 11, 15, 7, 8, 19, 5, 16, 24, 7, 10, 15, 9, 14, 11, 13, 13, 13, 18, 34, 17, 17, 17, 21, 19, 35, 22, 25, 34, 28, 30, 29, 32, 26, 30, 32, 29, 31, 31, 33, 27, 35, 28, 33, 29, 31, 36, 36, 37, 37, 37\}

 \item \{4, 10, 14, 6, 9, 18, 4, 11, 15, 7, 8, 19, 5, 16, 28, 7, 10, 15, 9, 14, 11, 13, 13, 13, 18, 20, 17, 17, 17, 21, 19, 23, 22, 34, 26, 34, 30, 32, 25, 27, 25, 28, 30, 33, 31, 33, 29, 31, 29, 32, 35, 35, 36, 36, 37, 37, 37\}

 \item \{4, 8, 14, 6, 9, 15, 4, 7, 20, 7, 8, 16, 5, 16, 18, 7, 10, 9, 11, 13, 14, 13, 20, 13, 15, 18, 19, 22, 28, 22, 26, 30, 21, 26, 32, 21, 23, 27, 25, 34, 25, 28, 35, 27, 29, 36, 29, 31, 33, 35, 33, 36, 33, 34, 37, 37, 37\}

 \item \{4, 8, 14, 6, 9, 15, 4, 7, 20, 7, 8, 16, 5, 16, 18, 7, 10, 9, 11, 13, 14, 13, 20, 13, 15, 18, 19, 22, 17, 26, 34, 21, 24, 27, 21, 28, 23, 32, 25, 27, 26, 35, 28, 30, 29, 31, 29, 36, 33, 35, 33, 36, 33, 34, 37, 37, 37\}

 \item \{4, 8, 14, 6, 9, 15, 4, 7, 20, 7, 8, 16, 5, 16, 18, 7, 10, 9, 11, 13, 14, 13, 20, 13, 15, 18, 19, 22, 17, 28, 34, 21, 26, 34, 21, 25, 24, 29, 25, 29, 30, 27, 31, 27, 30, 32, 32, 31, 33, 33, 35, 35, 36, 36, 37, 37, 37\}

 \item \{4, 8, 14, 6, 9, 15, 4, 7, 32, 7, 8, 16, 5, 16, 18, 7, 10, 9, 11, 13, 14, 13, 32, 13, 15, 18, 19, 21, 17, 22, 24, 19, 20, 25, 34, 23, 34, 26, 30, 26, 29, 29, 31, 27, 28, 35, 31, 35, 30, 36, 36, 33, 33, 33, 37, 37, 37\}

 \item \{4, 10, 14, 6, 9, 16, 4, 11, 20, 5, 7, 8, 5, 12, 7, 10, 15, 14, 22, 11, 19, 13, 13, 13, 16, 17, 17, 18, 17, 24, 26, 23, 30, 21, 24, 27, 34, 23, 32, 28, 29, 27, 29, 35, 27, 36, 29, 31, 33, 35, 33, 36, 33, 34, 37, 37, 37\}

 \item \{4, 10, 14, 6, 9, 16, 4, 11, 20, 5, 7, 8, 5, 12, 7, 10, 15, 14, 22, 11, 21, 13, 13, 13, 16, 17, 17, 18, 17, 24, 30, 21, 26, 34, 23, 25, 25, 26, 35, 24, 29, 33, 27, 31, 31, 33, 29, 30, 32, 35, 36, 36, 33, 34, 37, 37, 37\}

 \item \{4, 10, 14, 6, 9, 16, 4, 11, 20, 5, 7, 8, 5, 12, 7, 10, 15, 14, 22, 11, 32, 13, 13, 13, 16, 17, 17, 18, 17, 28, 34, 22, 26, 30, 25, 27, 25, 26, 32, 33, 31, 33, 34, 25, 28, 35, 29, 31, 36, 29, 36, 31, 35, 33, 37, 37, 37\}

 \item \{4, 8, 10, 6, 9, 14, 4, 7, 11, 7, 8, 17, 15, 15, 20, 28, 7, 34, 9, 35, 13, 34, 13, 35, 13, 16, 36, 32, 36, 32, 18, 21, 19, 28, 24, 30, 22, 25, 23, 24, 23, 31, 26, 30, 26, 27, 27, 31, 29, 29, 29, 33, 33, 33, 37, 37, 37\}

 \item \{4, 10, 34, 6, 9, 16, 4, 11, 35, 5, 7, 8, 15, 24, 36, 7, 10, 35, 9, 34, 11, 13, 13, 17, 19, 28, 32, 16, 22, 29, 20, 36, 17, 30, 21, 24, 31, 22, 31, 21, 26, 23, 25, 25, 30, 27, 27, 27, 29, 29, 32, 33, 33, 33, 37, 37, 37\}

\end{enumerate}

\end{document}